\documentclass[a4paper,11pt,twoside]{amsart}
\usepackage[english]{babel}
\usepackage[utf8]{inputenc}

\usepackage[a4paper,inner=3cm,outer=3cm,top=4cm,bottom=4cm,pdftex]{geometry}
\usepackage{fancyhdr}
\usepackage{titlesec}
\usepackage{color}
\usepackage{bold-extra}
\usepackage{ mathrsfs }
\titleformat{\section}{\normalfont\scshape\centering}{\thesection}{1em}{}
  \titleformat{\subsection}{\bfseries}{\thesubsection}{1em}{}
  
\usepackage{comment}
\usepackage{graphics}
\usepackage{aliascnt}
\usepackage[pdftex,citecolor=green,linkcolor=red]{hyperref}

\usepackage{amsmath}
\usepackage{amsfonts}
\usepackage{amssymb}
\usepackage{amsthm}
\usepackage{comment}
\usepackage{mathtools}

\newtheorem{theorem}{Theorem}[section]
\newtheorem{corollary}[theorem]{Corollary}

\newtheorem{lemma}[theorem]{Lemma}
\newtheorem{proposition}[theorem]{Proposition}
\theoremstyle{definition}
\newtheorem{definition}[theorem]{Definition}
\newtheorem{remark}[theorem]{Remark}
\newtheorem{conjecture}[theorem]{Conjecture}
\numberwithin{equation}{section}
\newcommand{\sgn}{\text{sgn}}
\newcommand\eps{\varepsilon}
\newcommand\plim{\mathop{\widetilde \lim}}

\newcommand\E{\mathbb{E}}
\newcommand\R{\mathbb{R}}
\newcommand\Z{\mathbb{Z}}
\newcommand\N{\mathbb{N}}
\newcommand\C{\mathbb{C}}

\newcommand{\Mod}[1]{\ (\mathrm{mod}\ #1)}

\begin{document}
\title{Value patterns of multiplicative functions and related sequences}

\date{}
\author{Terence Tao}
\address{Department of Mathematics, UCLA\\
405 Hilgard Ave\\
Los Angeles CA 90095\\
USA}
\email{tao@math.ucla.edu}

\author{Joni Ter\"av\"ainen}
\address{Mathematical Institute, University of Oxford\\ Radcliffe Observatory Quarter, Woodstock Rd\\
Oxford OX2 6GG\\
UK}
\email{joni.teravainen@maths.ox.ac.uk}

\begin{abstract}
We study the existence of various sign and value patterns in sequences defined by multiplicative functions or related objects. For any set $A$ whose indicator function is ``approximately multiplicative'' and uniformly distributed on short intervals in a suitable sense, we show that the density of the pattern $n+1\in A$, $n+2\in A$, $n+3\in A$ is positive, as long as $A$ has density greater than $\frac{1}{3}$. Using an inverse theorem for sumsets and some tools from ergodic theory, we also provide a theorem that deals with the critical case of $A$ having density exactly $\frac{1}{3}$, below which one would need nontrivial information on the local distribution of $A$ in Bohr sets to proceed. We apply our results firstly to answer in a stronger form a question of Erd\H{o}s and Pomerance on the relative orderings of the largest prime factors $P^{+}(n)$, $P^{+}(n+1), P^{+}(n+2)$ of three consecutive integers. Secondly, we show that the tuple $(\omega(n+1),\omega(n+2),\omega(n+3)) \Mod 3$ takes all the $27$ possible patterns in $(\mathbb{Z}/3\mathbb{Z})^3$ with positive lower density, with $\omega(n)$ being the number of distinct prime divisors. We also prove a theorem concerning longer patterns $n+i\in A_i$, $i=1,\dots k$ in approximately multiplicative sets $A_i$ having large enough densities, generalising some results of Hildebrand on his ``stable sets conjecture''. Lastly, we consider the sign patterns of the Liouville function $\lambda$ and show that there are at least $24$ patterns of length $5$ that occur with positive upper density. In all of the proofs we make extensive use of recent ideas concerning correlations of multiplicative functions.  
 \end{abstract}
\maketitle

\section{Introduction}

For any function $a\colon \mathbb{N}\to S$ with finite range $S$ and any $k\in \mathbb{N}\coloneqq \{1,2,\ldots\}$, we may define the length $k$ \emph{value patterns} of $a$ to be the tuples $s\in S^k$ that are of the form
\begin{align*}
s = (a(n+1),a(n+2),\ldots,a(n+k)) 
\end{align*}
for some\footnote{One could also include the $n=0$ case here if one wished, although it will not affect our main results.} $n\in \mathbb{N}$. We further say that the function $a$ \emph{attains a pattern $s$ with positive lower density} (resp. \emph{upper density}) if the set
\begin{align*}
\{n\in \mathbb{N}: (a(n+1),a(n+2),\ldots, a(n+k))=s\}    
\end{align*}
has positive lower density\footnote{For the precise definitions of the various densities used in this paper, as well as the standard arithmetic functions and asymptotic notation, see Subsection \ref{sub: not}.} (resp. upper density). In the case $S=\{-1,+1\}$, we will refer to value patterns as \emph{sign patterns}. In this paper, we will mostly be interested in whether or not a given pattern is attained with positive lower density.  

The occurrence of various value patterns for an arithmetic function $a$ has attracted particular interest in the case where $a \colon \mathbb{N}\to \mathbb{D}$ is \emph{multiplicative}, that is to say $a(1)=1$ and $a(mn)=a(m)a(n)$ whenever $m$ and $n$ are coprime natural numbers. Here $\mathbb{D}\coloneqq \{z\in \mathbb{C}:|z|\leq 1\}$ is the unit disc of the complex plane. Indeed, the interaction of multiplicative functions with their shifts is the subject of many conjectures, including those of Chowla \cite{chowla} and Elliott \cite{elliott-book}, \cite{mrt-average}. In particular, for the Liouville function $\lambda(n)$ and the M\"{o}bius function $\mu(n)$ the existence of various sign or value patterns has been actively studied, due to connections to the aforementioned conjectures.  

Chowla's conjecture \cite{chowla} for the Liouville function states that the autocorrelations\footnote{Unless otherwise stated, all variables such as $n$ appearing in summations are understood to be restricted to the natural numbers, with the exception of variables named $p$ (or $p_1$, $p_2$, etc.) which are understood to be restricted to the primes.} 
$$ \frac{1}{x} \sum_{n \leq x} \lambda(n+h_1) \dots \lambda(n+h_k) $$
of the Liouville function $\lambda$ converge to $0$ as $x \to \infty$, for any $k \geq 1$ and distinct natural numbers $h_1,\dots,h_k$.  This conjecture easily implies that $\lambda(n)$ attains all the $2^k$ sign patterns in $\{-1,+1\}^k$ for any $k$ infinitely often, and in fact the conjecture is equivalent to each of these length $k$ patterns occurring with asymptotic density $2^{-k}$. The analogous version of Chowla's conjecture for the M\"obius function\footnote{By using the identities $\lambda(n) = \sum_{d^2|n} \mu(n/d^2)$ and $\mu(n) = \sum_{d^2|n} \mu(d) \lambda(n/d^2)$ and exploiting the absolute convergence of the sum $\sum_d \frac{1}{d^2}$, one can easily show that the Chowla conjectures for the Liouville and M\"obius functions are equivalent if one generalises the distinct linear forms $n+h_1,\dots,n+h_k$ to non-parallel affine forms $a_1 n + h_1,\dots,a_k n+h_k$; we omit the details.} $\mu$ implies that the function $\mu$ attains every \emph{admissible} value pattern in $(\varepsilon_1,\ldots, \varepsilon_k)\in\{-1,0,+1\}^k$ infinitely often, where we call a pattern admissible if for every prime $p$ there exists $b\in [0,p^2-1]$ such that $\varepsilon_{p^2j+b}=0$ for all $j$ satisfying $1\leq p^2j+b\leq k$. Nevertheless, Chowla's conjecture (for either $\lambda$ or $\mu$) remains unsolved once $k \geq 2$, and thus these implications are only conditional. In Subsection \ref{sub: Liouville}, we will give an account of the unconditional results on sign patterns of the Liouville function, as well as state our new result on length $5$ patterns.\\

In this paper, we study the appearance of value patterns in more sequences that have some multiplicative structure. Let $f \colon \mathbb{N}\to \mathbb{D}$ be a completely multiplicative function\footnote{We say that $f$ is \emph{completely multiplicative} if $f(mn)=f(m)f(n)$ for all $m,n\in \mathbb{N}$ and $f(1)=1$.}, and assume that the range $f(\mathbb{N})$ is a finite set, so that it is meaningful to talk about the sign patterns of $f$. Then actually $f(\mathbb{N})=\mu_m$ or $f(\mathbb{N})=\mu_m\cup\{0\}$ for some $m$, where $\mu_m \coloneqq \{ z \in \C: z^m=1\}$ is the set of roots of unity of order $m$. The case of $f(\mathbb{N})=\{-1,+1\}$, is rather similar to the case of the Liouville function $\lambda$, and in fact it follows easily\footnote{In \cite{tt-elliott}, the proof was written only for $f=\lambda$, but the exact same argument works for any completely multiplicative bounded $f$ that is not weakly pretentious.} from \cite[Corollary 1.6; Proof of Corollary 7.2]{tt-elliott} that if $f$ is \emph{not weakly pretentious}, by which we mean that
\begin{align*}
\sum_{p\leq x}\frac{1-\textnormal{Re}(f(p)\bar{\chi}(p))}{p}\gg_{\chi}\log \log x    
\end{align*}
for any Dirichlet character $\chi$, then $f$ attains all the $16$ possible length $4$ sign patterns with positive lower density. At the opposite extreme, the case of $f$ being \emph{pretentious} in the sense that
\begin{align*}
\sum_{p\leq x}\frac{1-\textnormal{Re}(f(p)\bar{\chi}(p))}{p}\ll 1 
\end{align*}
for some Dirichlet character $\chi$ was recently considered by Klurman and Mangerel \cite{klurman-mangerel-gowers}. We also remark that if $f(\mathbb{N})\subset \mu_m$ and if $f$ satisfies the non-pretentiousness condition
\begin{align*}
\sum_{p\leq x}\frac{1-\textnormal{Re}(f(p)^d\bar{\chi}(p))}{p}\xrightarrow{x\to \infty}\infty    
\end{align*}
for all $1\leq d\leq m-1$ and every Dirichlet character $\chi$, then Elliott's conjecture \cite{elliott-book}, \cite{mrt-average} on correlations of multiplicative functions would imply\footnote{Indeed, if we use the expansion $1_{f(n)=e(a/m)}=\frac{1}{m}\sum_{j=0}^{m-1}f(n)^j e(-aj/m)$, we immediately reduce the study of the value patterns of $f$ to bounding its correlations, which can be shown to be negligible assuming Elliott's conjecture.} that $f$ attains every value pattern in $\mu_m^k$ with equal asymptotic density $m^{-k}$. For $k=2$ (and if one uses logarithmic density instead of asymptotic density), this follows unconditionally from \cite[Theorem 1.5]{tao-chowla}, and from \cite[Corollary 1.6]{tt-elliott} we can deduce various special cases for higher values $k\geq 3$ (again using logarithmic density in place of asymptotic density).\\ 

In what follows, we will mostly be studying the case $f(\mathbb{N})=\{0,1\}$, and only make the weaker assumption that $f$ is ``approximately multiplicative'' in a precise sense defined in Subsection \ref{sub: stable} (there we call this notion of approximate multiplicativity ``weak stability''). In this case, it is natural to write $f(n)=1_A(n)$ for some set $A\subset \mathbb{N}$ and to say that the set $A$ itself is ``approximately multiplicative''. The occurrence of patterns in such sets is not covered by Elliott's conjecture. It turns out that the class of \emph{genuinely} multiplicative sets of positive asymptotic density are not a particularly interesting class of sets (a typical example being the set $\{n: \mu^2(n)=1\}$ of square-free numbers, the patterns of which are well-understood from basic sieve theory), but the wider class of \emph{approximately} multiplicative sets instead does include various interesting sets related to the largest prime factors of integers or to the number of prime divisors of an integer.  For instance, if $P^{+}(n)$ denotes the largest prime factor of a natural number $n$ (and $P^{+}(1)\coloneqq 1$), the sets 
\begin{align}\label{eq00}
Q_{\alpha, \beta}\coloneqq \{n\in \mathbb{N}: n^{\alpha}<P^{+}(n)<n^{\beta}\}    
\end{align}
with $0\leq \alpha < \beta \leq 1$ turn out to be sufficiently close to being multiplicative that our results in Subsection \ref{sub: stable} apply, and we will present several applications of our results to patterns in the sets $Q_{\alpha,\beta}$. See also Subsections \ref{sub: largest} and \ref{sub: omega} below for more applications of our results to value patterns of approximately multiplicative sets.

We also investigate the case $f(\mathbb{N})=\mu_3$, and more specifically the case $f(n) \coloneqq e(\frac{\omega(n)}{3})$, where $\omega(n)$ is the number of prime factors of $n$ without multiplicities, and $e(\theta) \coloneqq e^{2\pi i \theta}$. In this case, the prior knowledge on length $3$ value patterns was very limited, since the fact that $f^3=1$ makes the result in \cite[Corollary 1.6]{tt-elliott} on $3$-point correlations of multiplicative functions inapplicable. The functions $n\mapsto e(\frac{\omega(n)}{q})$ can be thought of as generalisations of the Liouville or M\"{o}bius functions\footnote{For $q=2$, the function $n\mapsto (-1)^{\omega(n)}$ is of course not quite equal to either the Liouville function or the M\"{o}bius function, but is very closely connected to both since it takes the value $-1$ at all the primes.} which takes values in the $q$th roots of unity rather than in $\{-1,+1\}$, and their value patterns are in one-to-one correspondence with those of the sequence $\omega(n)\Mod q$.

Before stating our results on patterns in general approximately multiplicative sets, we state the corollaries of our results for the sets $Q_{\alpha, \beta}$ and $\{n\in \mathbb{N}: \omega(n)\Mod 3\}$ mentioned above.

\subsection{Comparison of largest prime factors of consecutive integers} \label{sub: largest}

In what follows, let
\begin{align*}
d_{-}(A) \coloneqq \liminf_{x\to \infty}\frac{|A\cap [1,x]|}{x}    
\end{align*}
denote the lower density of a set $A\subset \mathbb{N}$. 

In 1978, Erd\H{o}s and Pomerance \cite{ep} studied the orderings of the largest prime factors of consecutive integers, and showed that
\begin{align}\label{eq0}
d_{-}(\{n\in \mathbb{N}: P^{+}(n+1)<P^{+}(n+2)\})\geq c_0>0    
\end{align}
for some explicit $c_0$. They also showed that the set 
\begin{align}\label{eq1}
\{n\in \mathbb{N}: P^{+}(n+1)<P^{+}(n+2)<P^{+}(n+3)\} 
\end{align}
is infinite by looking at the explicit sequence $n=p^{2^{k_p}}-2$ with $k_p$ suitably chosen for every odd $p$, and raised the problem of proving that also the set  
\begin{align}\label{eq2}
\{n\in \mathbb{N}: P^{+}(n+1)>P^{+}(n+2)>P^{+}(n+3)\} 
\end{align}
corresponding to the opposite ordering was infinite. This was eventually solved by Balog \cite{balog}, who showed that there are infinitely many solutions having the specific form $n=m^2-2$. It is clear however that both the construction of Erd\H{o}s and Pomerance and that of Balog only produce sparse sequences of $n\leq x$ that belong to the sets \eqref{eq1} or \eqref{eq2}; for \eqref{eq1} we get $\ll \sqrt{x}$ elements up to $x$ (since certainly we must have $k_p\geq 1$), and for \eqref{eq2} Balog's proof gives $\asymp \sqrt{x}$ elements up to $x$. 

Our main theorem in Subsection \ref{sub: stable} will be seen in Section \ref{sec: apps} to imply the following strengthenings of the above results, in which the sets \eqref{eq1}, \eqref{eq2} are shown to have positive lower density, and and also give some limited comparison with $P^+(n+4)$, or with various powers $n^\alpha, n^\beta$:

\begin{theorem}[Orderings of largest prime factors]\label{theo_comparison} We have
\begin{align*}
d_{-}(\{n\in \mathbb{N}: P^{+}(n+1)<P^{+}(n+2)<P^{+}(n+3)>P^{+}(n+4)\})>0    
\end{align*}
and
\begin{align*}
d_{-}(\{n\in \mathbb{N}: P^{+}(n+1)>P^{+}(n+2)>P^{+}(n+3)<P^{+}(n+4)\})>0.        
\end{align*}
\end{theorem}

\begin{theorem}[Largest prime factors of three consecutive integers]\label{theo_largest} Let $0<\alpha<\beta<1$ be real numbers, such that $\rho(1/\alpha) + \rho(1/\beta) \neq 1$, where $\rho$ is the Dickman function (see \cite{ht}). Then we have
\begin{align*}
d_{-}(\{n\in \mathbb{N}: P^{+}(n+1)<n^{\alpha}<P^{+}(n+2)<n^{\beta}< P^{+}(n+3)\})>0    
\end{align*}
and 
\begin{align*}
d_{-}(\{n\in \mathbb{N}: P^{+}(n+3)<n^{\alpha}<P^{+}(n+2)<n^{\beta}< P^{+}(n+1)\})>0.    
\end{align*}
\end{theorem}

As mentioned above, either of Theorem \ref{theo_comparison} and Theorem \ref{theo_largest} immediately imply the new result that the sets \eqref{eq1}, \eqref{eq2} both have positive lower density.  The condition $\rho(1/\alpha) + \rho(1/\beta) \neq 1$ should be removable, but this seems to be beyond the methods in this paper (unless there is substantial progress on understanding local Fourier uniformity of multiplicative functions or indicator functions of weakly stable sets).

We remark that the study of the largest prime factors of \emph{two} consecutive integers has been taken up by several authors. In particular, the original value of $c_0=0.0099$ in \eqref{eq0} by Erd\H{o}s and Pomerance was improved by de la Bret\`eche, Pomerance and Tenenbaum \cite{bpt} to $c_0=0.05544$, and the current record is held by Wang \cite{wang2} with $c_0=0.1356$. It was conjectured in the correspondence of Erd\H{o}s and Tur\'an \cite{erdos-turan} (and repeated by Erd\H{o}s in \cite{erdos-conj}) that the set of $n$ with $P^{+}(n)<P^{+}(n+1)$ has asymptotic density equal to $1/2$, as one would naturally expect. In \cite{tera-binary}, it was shown that the \emph{logarithmic} density of this set indeed equals $1/2$. For orderings of longer strings of consecutive values of $P^{+}(n)$, little is known, but Wang \cite{wang2} showed that either of
\begin{align*}
P^{+}(n+i)<\min_{\substack{j\leq J\\j\neq i}}P^{+}(n+j)\quad \textnormal{and}\quad  P^{+}(n+i)>\max_{\substack{j\leq J\\j\neq i}}P^{+}(n+j)   
\end{align*}
happens with positive lower density for any $J\geq 3$. For completely arbitrary orderings of largest prime factors at consecutive integers, there is a natural conjecture of de Koninck and Doyon \cite{dd}, which states that for any permutation $\{a_1,\ldots, a_k\}$ of $\{1,\ldots, k\}$ we have 
\begin{align*}
d(\{n\in \mathbb{N}: P^{+}(n+a_1)<\cdots <P^{+}(n+a_k)\})=\frac{1}{k!}.    
\end{align*}
This however seems to be far out of reach, and even for $k=2$ we only know lower bounds for the asymptotic density and we know the correct value for the logarithmic density but do not know that the asymptotic density exists to start with.

\subsection{Patterns of the number of prime factors modulo 3}\label{sub: omega}

In Section \ref{sec: apps}, we will also utilise our main theorem stated in Subsection \ref{sub: stable} to prove the following result about the sign patterns of $\omega(n) \Mod 3$.

\begin{theorem}[Value patterns of $\omega \Mod{3}$]\label{theo_omega} The function $\omega(n) \Mod 3$ attains each of the $27$ possible length three value patterns with positive lower density. In other words, we have
\begin{align*}
d_{-}(\{n\in \mathbb{N}: \omega(n+1)\equiv a \Mod{3}, \omega(n+2)\equiv b\Mod{3}, \omega(n+3)\equiv c \Mod 3\})>0. \end{align*}
for all $a,b,c\in \mathbb{Z}/3\mathbb{Z}$.
The same holds for $\Omega(n)$, the number of prime factors of $n$ counting multiplicities, in place of $\omega(n)$.
\end{theorem}

The value patterns of $\Omega(n) \Mod 2$ have of course been an active subject of study, since they are in one-to-one correspondence with sign patterns of the Liouville function; see \cite{hildebrand-liouville}, \cite{mrt-sign}, \cite{tt-elliott} for some works studying the number of these sign patterns. Showing that $\Omega(n) \Mod 3$ attains all the value patterns of length three with positive lower density is evidently harder than showing the same for $\Omega(n) \Mod 2$ (which was shown by Matom\"aki, Radziwi\l{}\l{} and the first author in \cite{mrt-sign}), since the number of possible patterns for $\Omega(n) \Mod 3$ is $27$, meaning that each pattern should conjecturally have a rather small asymptotic density of $1/27$, as opposed to the much larger asymptotic density of $1/8$ corresponding to the patterns of length three for $\Omega(n) \Mod 2$. Perhaps surprisingly, it is much easier to deal with the longer patterns $(\Omega(n+1) \Mod{q_{1}},\ldots, \Omega(n+k) \Mod{q_k})$ for various choices of \emph{distinct} $q_j$. Namely, if $q_1,\ldots, q_k$ are all pairwise coprime, the authors showed in \cite[Theorem 1.13]{tt-elliott} that each of the $q_1\cdots q_k$ possible patterns occurs with logarithmic density $\frac{1}{q_1\cdots q_k}$. The fact that the patterns with coprime $q_j$ are easier stems from the result towards the Elliott conjecture in \cite{tt-elliott}, which applies to correlations
\begin{align}\label{eq3}
\frac{1}{\log x}\sum_{n\leq x}\frac{g_1(n+1)\cdots g_k(n+k)}{n}    
\end{align}
of $1$-bounded multiplicative functions whenever the \emph{product} $g_1\cdots g_k$ is not ``weakly pretentious''. If we expand $1_{\Omega(n)\equiv a_j \Mod{q_j}}$ as a linear combination of the multiplicative functions $n\mapsto e(\frac{b\Omega(n)}{q_j})$, then the logarithmic density of the sign pattern can be written as a linear combination of correlations like \eqref{eq3}, but without the assumption of the $q_j$ being corime the result in \cite[Corollary 1.6]{tt-elliott} on the correlations \eqref{eq3} is not directly applicable. Of course, assuming the full Elliott conjecture and applying the same strategy, one would see that each of the $q_1\cdots q_k$ value patterns is attained with asymptotic density $1/(q_1\cdots q_k)$ without any restrictions on the $q_j$, but, needless to say, even for $q_1=\cdots =q_k=2$ proving this is out of reach.

\subsection{Results on weakly stable sets}\label{sub: stable}

Theorems \ref{theo_comparison}, \ref{theo_largest} and \ref{theo_omega} will all be deduced from our main results concerning patterns in sets that are ``approximately multiplicative'' in a suitable sense. The notion of approximate multiplicativity that we want to consider is called stability. In what follows, we use the expectation notation 
\begin{align*}
\mathbb{E}_{n\in A}f(n)\coloneqq \frac{1}{|A|}\sum_{n\in A}f(n)    
\end{align*}
for any finite, nonempty set $A\subset \mathbb{N}$ and for any function $f:A\to \mathbb{C}$.

\begin{definition}\label{def_stable} \cite{balog-stable} We say that a set $A\subset \mathbb{N}$ is \emph{stable} if for every prime $p$ we have
\begin{align*}
    \lim_{x\to \infty} \mathbb{E}_{n\leq x}|1_{A}(n)-1_{A}(pn)|=0.
\end{align*}
Equivalently, $A$ is stable if and only if $d(A\triangle p^{-1}A)=0$ for every prime $p$, where $\triangle$ denotes the symmetric difference, and $p^{-1} A \coloneqq \{n \in \N: pn \in A \}$.
\end{definition}

An important class of stable sets is given by
\begin{align*}
Q_{\alpha, \beta}\coloneqq \{n\in \mathbb{N}: n^{\alpha}<P^{+}(n)<n^{\beta}\},    
\end{align*}
where $0\leq \alpha<\beta\leq 1$. By the classical result of Dickman \cite{dickman}, this set has asymptotic density $\rho(1/\beta)-\rho(1/\alpha) > 0$, where $\rho$ is the Dickman function.  The stability of $Q_{\alpha,\beta}$ then follows easily from the continuity of the Dickman function.

A completely different class of stable sets is
\begin{align*}
A_{\alpha,\beta}\coloneqq \left\{n\in \mathbb{N}: \frac{\omega(n)-\log \log n}{\sqrt{\log \log n}}\in [\alpha, \beta]\right\}    
\end{align*}
for $-\infty<\alpha<\beta<\infty$. By the Erd\H{o}s-Kac theorem, this set has a positive asymptotic density as well.

Stable sets were first introduced by Balog in \cite{balog-stable}, where he conjectured that if $A\subset \mathbb{N}$ is stable with $d_{-}(A)>0$, then the pattern $n+1\in A, n+2\in A$ occurs with positive lower density, or equivalently that
$$ d_-( (A-1) \cap (A-2) ) > 0.$$
This conjecture was settled by Hildebrand \cite{hildebrand-balog} using an elementary but intricate method. Hildebrand \cite{hildebrand-conj} himself later posed a conjecture that generalises Balog's conjecture to length $k$ patterns.

\begin{conjecture}[Hildebrand's stable sets conjecture \cite{hildebrand-conj}]\label{conj_hildebrand}  Let $k\geq 2$, and let $A\subset \mathbb{N}$ be any stable set with $d_{-}(A)>0$. Then we have 
\begin{align*}
d_{-}((A-1)\cap (A-2)\cap\cdots \cap(A-k))>0.    
\end{align*}
\end{conjecture}

For higher values of $k$, Conjecture \ref{conj_hildebrand} is certainly a deep one, since it implies for any $\varepsilon>0$ that both of the sets 
\begin{align}\label{eq4}
\{n\in \mathbb{N}: P^{+}(n+j)<n^{\varepsilon}\,\textnormal{for all}\, 1\leq j\leq k)\}    
\end{align}
and
\begin{align}\label{eq5}
\{n\in \mathbb{N}: P^{+}(n+j)>n^{1-\varepsilon}\, \textnormal{for all}\, 1\leq j\leq k)\}    
\end{align}
have positive lower density. Remarkably, Balog and Wooley \cite{balog-wooley} were able to prove that the set \eqref{eq4} is always infinite, but their construction gives a very sparse set of such $n$. For the set \eqref{eq5}, in turn, it is not even known that it is infinite, except for $k=2$ (which follows from \cite{hildebrand-balog}). 

It follows from a trivial pigeonholing argument that the stable sets conjecture holds when $d_-(A) > 1 - \frac{1}{k}$.  
Hildebrand \cite{hildebrand-higher} extended this range to $d_{-}(A)>1-\frac{1}{k-1}$ when $k \geq 3$; thus for instance he established the $k=3$ case of the conjecture for $d_{-}(A)>\frac{1}{2}$.

We make progress on a variant of the stable sets conjecture for all $k\geq 3$, where we have a somewhat different set of assumptions. Firstly, our theorem applies to $k$ \emph{distinct} sets $A_1,\ldots, A_k\subset \mathbb{N}$, whereas the method of Hildebrand in \cite{hildebrand-balog} appears difficult to adapt to this setting. Secondly, the notion of stability that we need is weaker than in Definition \ref{def_stable}; see Definition \ref{def_weakstable} below. On the other hand, we need a stronger density assumption for the $A_i$. It turns out that a stable set is always uniformly distributed in arithmetic progressions in the sense that 
\begin{align*}
d_{-}(A\cap \{n\in \mathbb{N}: n\equiv b\Mod q\})=\frac{1}{q}d_{-}(A)    
\end{align*}
for any $b,q\in \mathbb{N}$; see \cite{hildebrand-higher}. What we need in our main theorem is that a similar statement holds when $A$ is restricted to almost all short intervals. In all of our applications, this stronger condition will be satisfied by the Matom\"aki--Radziwi\l{}\l{} theorem \cite{mr-annals} or some variant thereof. 

We now define the precise concepts that we need for the main theorem.

\begin{definition}[Weakly stable sets]\label{def_weakstable} We say that a set $A\subset \mathbb{N}$ is \emph{weakly stable} if for every $x\geq 1$ there is a set $B_x\subset \mathbb{N}$ such that for every prime $p$ we have
\begin{align}\label{eq4a}
 \lim_{x\to \infty} \mathbb{E}_{\substack{n\leq x\\ p\nmid n}}|1_{A}(n)-1_{B_x}(pn)|=0.   
\end{align}
In addition, we say that the sequence $(B_x)$ \emph{corresponds to} $A$.
\end{definition}

It is clear that if $A$ is stable, then $A$ is also weakly stable (with $B_x = A$ in this case). Importantly for us, the class of weakly stable sets also contains interesting sets that do not satisfy the usual definition of stability; for example, the sets $A = \{n\in \mathbb{N}: \omega(n)\equiv a \Mod q\}$ are weakly stable but not stable for any $a\in \mathbb{N}$, $q\geq 2$; the point is that the sets $B_x$ need to be taken here to equal $\{n\in \mathbb{N}: \omega(n)\equiv a+1 \Mod q\} \neq A$. It is because of applications to such sets that we want to have the condition $p\nmid n$ in \eqref{eq4a}; without that condition, these sets would not be weakly stable. 

Another definition that we need is that of uniform distribution in short intervals. 

\begin{definition}[Uniform distribution in short intervals] We say that a set $A\subset \mathbb{N}$ is \emph{uniformly distributed in short intervals with asymptotic density} $\delta$ if we have
\begin{align*}
\lim_{H\to \infty}\limsup_{x\to \infty}\frac{1}{x}\int_{0}^{x}\left|\frac{|A\cap [y,y+H]\cap (q\mathbb{Z}+b)|}{H}-\frac{\delta}{q}\right|\, dy=0
\end{align*}
for all $b,q\in \mathbb{N}$.
\end{definition}

With this notation, we can prove the following results.

\begin{theorem}[$k=3$ main theorem, large density]\label{theo1} Let $A_1,A_2,A_3\subset \mathbb{N}$ be weakly stable and uniformly distributed in short intervals with densities $\delta_1,\delta_2,\delta_3 > 0$, respectively. Suppose that $\delta_1+\delta_2+\delta_3> 1$. Then
\begin{align*}
d_{-}((A_1-1)\cap (A_2-2) \cap (A_3-3))>0.    
\end{align*}
\end{theorem}

\begin{theorem}[$k=3$ main theorem, critical density]\label{theo1.5} Let $A_1,A_2,A_3\subset \mathbb{N}$ be weakly stable and uniformly distributed in short intervals with densities $\delta_1,\delta_2,\delta_3 > 0$, respectively. Suppose that $\delta_1+\delta_2+\delta_3=1$. Then for every $c\in \{0,1,2\}$ we have
\begin{align*}
d_{-}\left(\bigcup_{\substack{c_1,c_2,c_3\in \{0,1,2\}\\c_1+c_2+c_3\equiv c \Mod{3}}} (A_{c_1}-1)\cap (A_{c_2}-2) \cap (A_{c_3}-3)\right)>0.    
\end{align*}
Further, if $\delta_1\neq \delta_3$ and $d(A_1\cup A_2 \cup A_3)=1$, then
\begin{align*}
d_{-}((A_1-1)\cap (A_2-2)\cap (A_3-3))>0.    
\end{align*}
\end{theorem}

\begin{theorem}[$k>3$ main theorem]\label{theo2} Let $k\geq 4$, and let $A_1,\ldots, A_k\subset \mathbb{N}$ be weakly stable and uniformly distributed in short intervals with densities $\delta_1,\ldots, \delta_k > 0$, respectively. Define the constants $c_k$ by
\begin{align*}
c_4&\coloneqq \frac{3+\sqrt{2}}{7}=0.6306\ldots\\
c_5&\coloneqq \frac{9+2\sqrt{6}}{19}=0.7315\ldots
\end{align*}
and more generally $c_k \in (0,1)$ is the largest root of the quadratic equation
\begin{align*}
\left(\frac{9}{2}\binom{k}{3}+(6-4a_k)\binom{k}{2}\right)(1-X)^2+(a_k^2-a_k)k(1-X)-a_k(a_k-1) = 0,
\end{align*}
where $a_k\coloneqq \lceil\frac{3k+2}{4}\rceil$. Suppose that $\delta_i> c_k$ for all $i\leq k$. Then
\begin{align*}
d_{-}((A_1-1)\cap (A_2-2)\cap \cdots \cap (A_k-k))>0.    
\end{align*}
\end{theorem}

\begin{remark}
Inspecting the proof of Theorem \ref{theo2} in Section \ref{sec: higher}, we see that it works equally well for $k=3$ with $c_3=1/3$. However, since this is a special case of Theorem \ref{theo1} (namely the case $\delta_1,\delta_2,\delta_3>1/3$), we confine ourselves to $k\geq 4$ in Theorem \ref{theo2}.
\end{remark}

We remark that a routine but tedious calculation yields the asymptotic
$$ c_k=1-\frac{1}{k-\frac{4}{3}+\eta_k}$$
where $\eta_k$ goes to zero as $k$ goes to infinity.  For instance, one can calculate
\begin{align*}
\eta_4 &= 0.04044\dots \\
\eta_5 &= 0.05808\dots \\
\eta_{10} &= 0.04143\dots \\
\eta_{100} &= 0.00435\dots\\
\eta_{1000} &= 0.00071\dots.
\end{align*}
The value of $c_k$ should be compared with the value $1-\frac{1}{k-1}$, which is the threshold in Hildebrand's result about Conjecture \ref{conj_hildebrand}. It turns out that our value of $c_k$ is smaller (or equivalently, that $\eta_k < 1/3$) for every $k \geq 4$.

When it comes to our applications stated as Theorems \ref{theo_largest} and \ref{theo_omega}, we want to apply our main theorems to the triples of sets 
\begin{align*}
\{n\in \mathbb{N}: P^{+}(n)<n^{\alpha}\}, \,\, \{n\in \mathbb{N}: n^{\alpha}<P^{+}(n)<n^{\beta}\},\,\, \{n\in \mathbb{N}: P^{+}(n)>n^{\beta}\}    
\end{align*}
or 
\begin{align*}
\{n\in \mathbb{N}: \omega(n)\equiv a \Mod 3\}, \,\, \{n\in \mathbb{N}: \omega(n)\equiv b \Mod 3\}, \,\, \{n\in \mathbb{N}: \omega(n)\equiv c \Mod 3\}.    
\end{align*}
In either case, the sum of the densities of these sets will be \emph{exactly} $1$, so we are in the critical case $\delta_1+\delta_2+\delta_3=1$ where Theorem \ref{theo1} no longer applies. It turns out that the case $\delta_1+\delta_2+\delta_3=1$ is much more delicate than the case $\delta_1+\delta_2+\delta_3>1$, since for $\delta_1+\delta_2+\delta_3<1$ our method based on the study of sumsets in abelian groups breaks down. In addition, as soon as $\delta_1+\delta_2+\delta_3\leq 1$, all the $A_i$ could theoretically be ``local Bohr sets'' in the sense that, for any slowly growing function $H=H(X)$ tending to infinity we would have 
\begin{align*}
A_i\cap [x,x+H]=\{n\in [x,x+H]: n\alpha_{i,x}\in U_i\}    
\end{align*}
for almost all $x$ and for some irrational numbers $\alpha_{i,x}\in \mathbb{R}/\mathbb{Z}$ and open sets $U_i\subset \mathbb{R}/\mathbb{Z}$ of measure $\delta_i$. Such sets are certainly uniformly distributed in short intervals, and it may happen that $(A_1+A_3)\cap 2A_2=\emptyset$ when $\delta_1+\delta_2+\delta_3<1$ (see  Remark \ref{rmk2}), so that certainly $(A_1-1) \cap (A_2-2) \cap (A_3-3)=\emptyset$. Of course, we do not expect any such sets to be stable, but even showing that such local Bohr sets cannot be linear combinations of multiplicative functions appears very difficult. Even in the special case of $A=\{n\in \mathbb{N}: \Omega(n)\equiv 0\Mod 2\}$, it has not been shown that $A$ does not correlate with local Bohr sets, as that would amount to showing that
\begin{align}\label{eq5a}
\frac{1}{X}\int_{X}^{2X}\sup_{\alpha \in \mathbb{R}}|\mathbb{E}_{x\leq n\leq x+H}\lambda(n)e(\alpha n)|\, dx=o(1)
\end{align}
for any $H=H(X)$ tending to infinity, which is the Fourier uniformity conjecture from \cite{tao-sarnak}. See however \cite{mrt-Fourier} for recent progress on this. The sup norm estimate \eqref{eq5a} is open for slowly growing functions $H=H(X)=X^{o(1)}$, and it is in fact closely connected to Chowla's conjecture (see \cite{tao-sarnak} for this connection).  Nevertheless, it is still possible to deploy tools from additive combinatorics to be able to establish results like Theorem \ref{theo1.5} (and hence Theorems \ref{theo_largest}, \ref{theo_omega}) even if the weakly stable sets involved behave like Bohr sets, thus allowing us to avoid having to establish unproven results such as \eqref{eq5a}.

Both Theorem \ref{theo1} and \ref{theo2} can be applied to the sets $Q_{\alpha, \beta}$ defined in \eqref{eq00}, and they yield the following results about the largest prime factors of consecutive integers.

\begin{theorem}[Consecutive triples with large prime factors]\label{theo_threeprimes} Let $\gamma_3\coloneqq e^{-1/3}=0.7165\ldots$. Then for any $\gamma<\gamma_3$ we have
\begin{align*}
d_{-}(\{n\in \mathbb{N}:\,\, P^{+}(n+1)> n^{\gamma}, P^{+}(n+2)> n^{\gamma}, P^{+}(n+3)> n^{\gamma}\})>0. 
\end{align*}
\end{theorem}

Here $\gamma_3$ is the solution to $1-\rho(1/x)=1/3$, so the set $\{n\in \mathbb{N}: P^{+}(n)>n^{\gamma_3}\}$ has asymptotic density $1/3$. In \cite{hildebrand-higher}, the same was proved with $\gamma_3$ replaced by the smaller value $e^{-1/2}=0.6065\ldots$, where this value of $\gamma_{3}$ solves $1-\rho(1/x)=1/2$. 

We can also prove a result for longer strings of largest prime factors.

\begin{theorem}[Consecutive $k$-tuples with large prime factors]\label{theo_kprimes} Define
\begin{align*}
\gamma_4 & \coloneqq 0.5322\\
\gamma_5 & \coloneqq 0.4804.
\end{align*}
Then for $k=4,5$, we have 
\begin{align*}
d_{-}(\{n\in \mathbb{N}:\,\, P^{+}(n+1)> n^{\gamma_k}, P^{+}(n+2)> n^{\gamma_k},\ldots, P^{+}(n+k)> n^{\gamma_k}\})>0.    
\end{align*}
\end{theorem}
Again, Hildebrand \cite{hildebrand-higher} proved a similar result with $\gamma_k$ replaced by the smaller value $\frac{1}{\rho^{-1}(\frac{1}{k-1})}$, where $\rho^{-1}$ is the inverse function of the Dickman $\rho$ function. Like his result, ours can also be applied for higher values of $k$, but since our value of $\gamma_k$ behaves asymptotically like Hildebrand's value as $k\to \infty$, we omit the cases $k\geq 6$ from the theorem. 

\subsection{Sign patterns of the Liouville function}\label{sub: Liouville}

In Section \ref{sec: signpattern}, we will prove a result on length five sign patterns of the Liouville function. This application will not be based on Theorems \ref{theo1} or \ref{theo2} but nevertheless, like those theorems, it will be reduced to results about for correlations of multiplicative functions. In particular, we will use what we called an ``isotopy formula'' in \cite[Section 1]{tt-elliott} that implies in particular that
\begin{align*}
\E_{n\leq x}^{\log} \lambda(n+h_1)\cdots \lambda(n+h_k)=\E_{n\leq x}^{\log} \lambda(n-h_1)\cdots \lambda(n-h_k)+o(1)     
\end{align*}
for any $h_1,\ldots, h_k\in \mathbb{N}$. We will use this to show that there are at least $24$ sign patterns of length $5$ for the Liouville function.

\begin{theorem}[Length five sign patterns of Liouville]\label{theo_sign} There are at least $24$ sign patterns in $\{-1,+1\}^5$ that are attained by $\lambda$ with positive upper density, including the six explicit sign patterns
$$\pm (+1,+1,+1,+1,-1), \pm (+1,+1,+1,-1,-1), \pm (+1,-1,+1,+1,-1) $$
and their reversals
$$\pm (-1,+1,+1,+1,+1), \pm (-1,-1,+1,+1,+1), \pm (-1,+1,+1,-1,+1) $$
\end{theorem}

If we denote by $s(k)$ the number of length $k$ sign patterns that occur infinitely often in the Liouville function, then Theorem \ref{theo_sign} implies that $s(5)\geq 24$. In \cite[Corollary 7.2]{tt-elliott}, the authors proved that $s(4)=16$. For large values of $k$, our knowledge on $s(k)$ is rather weak; \cite[Remark 1.12]{tt-elliott} gives the explicit bound $s(k)\geq 2k+8$, whereas Frantzikinakis and Host \cite[Theorem 1.2]{fh-sarnak} proved that $s(k)$ grows faster than linearly with $k$. Very recently, this was improved by McNamara \cite{redmond} to $s(k) \gg k^2$.  Trivially, if we had Chowla's conjecture, then $s(k)=2^k$ would follow.

In order to improve the bound of $24$ in Theorem \ref{theo_sign}, one would have to improve the known bounds on the correlations of the Liouville function. Namely, if we define 
\begin{align*}
C_{A}\coloneqq \lim_{m\to \infty}\mathbb{E}_{n\leq x_m}^{\log}\prod_{j\in A}\lambda(n+j)
\end{align*}
for any finite set $A\subset \mathbb{N}$, where the sequence $(x_m)$ tending to infinity is chosen so that all the limits exists (which is possible by a diagonal argument), then from \cite[Proposition 7.1]{tt-elliott} we have the bound $|C_{\{1,2,\ldots, k\}}|\leq 1/2$. If this bound was sharp for $k=4$, then we could have the hypothetical scenario
\begin{align*}
C_{\{1,2,3,4\}}=C_{\{2,3,4,5\}}=\frac{1}{2},\quad C_{\{1,2,3,5\}}=C_{\{1,2,4,5\}}=C_{\{1,2,3,5\}}=0,    
\end{align*}
in which case one would easily see (using the odd order logarithmic Chowla conjecture from \cite[Theorem 1.1(i)]{tt-elliott}) that there are no more than $24$ sign patterns of the Liouville function that occur with positive logarithmic lower density. Thus one would have to rule out this scenario to be able to improve on the number of length $5$ sign patterns. 

\subsection{Proof strategy}

We briefly describe the ideas that go into the proofs of Theorems \ref{theo1}, \ref{theo1.5} and \ref{theo2}. Consider for example Theorem \ref{theo1}. By an elementary argument one sees that $d_{-}((A_1-1)\cap (A_2-2)\cap(A_3-3))>0$ is equivalent to the triple correlation
\begin{align}\label{eq52}
\mathbb{E}_{x/\omega(x)\leq n\leq x}^{\log}1_{A_1}(n+1)1_{A_2}(n+2)1_{A_3}(n+3)
\end{align}
being $\gg 1$ as $x\to \infty$ for every $\omega(X)\leq X$ tending to infinity. The functions $1_{A_i}$ are not assumed to be  multiplicative, but the assumption of weak stability works as a useful substitute to this, since for some sets $B_{x,i}$ and all primes $p$ we can write $1_{A_i}(n)=1_{B_{x,i}}(pn)+o(1)$ for most $n\leq x$. Using this relation, averaging \eqref{eq52} over primes, and applying the entropy decrement argument from \cite{tao-corr}, \cite{tt-elliott}, we conclude that \eqref{eq52} equals to
\begin{align}\label{eq53}
\mathbb{E}_{p\leq P}^{\log}\mathbb{E}_{x/\omega(x)\leq n\leq x}^{\log}1_{B_{x,1}}(n+p)1_{B_{x,2}}(n+2p)1_{B_{x,3}}(n+3p)+o(1)
\end{align}
with $P=P(x)$ being a medium size parameter. Such a double average is evidently easier to analyze than a single average. The only information that we will use about the sets $B_{x,i}$ is that they are uniformly distributed in short intervals with densities $\delta_1,\delta_2,\delta_3>0$, respectively, as follows easily from the fact that the $A_i$ have this property.  

Appealing to the Furstenberg correspondence principle, the average \eqref{eq53} being $\gg 1$ will follow from the following ergodic-theoretic statement: For any measure-preserving system $(X,\mu,T)$ and any measurable sets $B_1,B_2,B_3\subset X$ satisfying the uniform distribution property
\begin{align}\label{eq56}
\lim_{H\to \infty}\int_{X}|\mathbb{E}_{h\leq H}1_{B_i}(T^{qh}x)-\delta_i|\, d\mu(x)=0
\end{align} 
for all $q\in \mathbb{N}$ and with $\delta_i$ as in Theorem \ref{theo1}, we have
\begin{align}\label{eq54}
\mathbb{E}_{p\leq P}^{\log}\int_{X}1_{B_1}(T^{p}x)1_{B_2}(T^{2p}x)1_{B_3}(T^{3p}x)\, d\mu(x)\gg 1.
\end{align}
By the generalised von Neumann theorem and the Gowers uniformity of the primes \cite{gt-linear}, the bound \eqref{eq54} will follow from 
\begin{align}\label{eq55}
\mathbb{E}_{d\leq P:\, (d,W)=1}^{\log}\int_{X}1_{B_1}(T^{d}x)1_{B_2}(T^{2d}x)1_{B_3}(T^{3d}x)\, d\mu(x)\gg 1,
\end{align} 
where we are now averaging over integers rather than primes and $W:=\prod_{p\leq w}p$ (with $w$ a slowly growing function of $P$). This is roughly the conclusion we reach after Section \ref{sec: correspondence}.

In Section \ref{sec: 3}, we make several ergodic-theoretic reductions to reduce to the case where $X=(\mathbb{R}/\mathbb{Z})^d\times (\mathbb{Z}/m\mathbb{Z})$ for some $d,m\in \mathbb{N}$, so that the problem has essentially been reduced to the same problem on a torus. Now we apply a Pollard-type inequality from \cite{tao-kemperman} (which can be viewed as a quantitative version of the inequality $\mu(A+B)\geq \mu(A)+\mu(B)$ valid for compact subsets $A,B\subset X$ of any compact, connected abelian group, with $\mu$ being the Haar measure on $X$) to conclude the proof (it is here that the assumption $\delta_1+\delta_2+\delta_3>1$ is crucial).

In the case of Theorem \ref{theo1.5}, we proceed similarly up to the point where $X=(\mathbb{R}/\mathbb{Z})^d\times (\mathbb{Z}/m\mathbb{Z})$. Since $\delta_1+\delta_2+\delta_3$ is exactly $1$, the Pollard-type inequality is no longer sufficient to conclude, but employing instead an inverse theorem for it from \cite{tao-kneser} (see Theorem \ref{invt}), we can deduce that \eqref{eq55} holds unless $B_1,B_2,B_3$ (or rather their projections to $(\mathbb{R}/\mathbb{Z})^d$) are essentially Bohr sets. The case where $B_1,B_2,B_3$ are Bohr sets can be dealt with a bit of Fourier analysis, and we eventually conclude that \eqref{eq55} holds then as well under the conditions of Theorem \ref{theo1.5}.

For Theorem \ref{theo2}, we make a similar reduction to the statement
\begin{align*}
\mathbb{E}_{p\leq P}^{\log}\int_{X} 1_{B_1}(T^p x)1_{B_2}(T^{2p}x)...1_{B_k}(T^{kp}x)\, d\mu(x)\gg 1
\end{align*}  
with the $B_i$ satisfying \eqref{eq56} as before. One easily sees from \eqref{eq56} that $\int_X 1_{B_i}(x)\, d\mu(x)=\delta_i$, $\int_X 1_{B_{i_1}}(T^{i_1 p} x)1_{B_{i_2}}(T^{i_2p}x)\, d\mu(x)=\delta_{i_1}\delta_{i_2}$ for $1\leq i_1<i_2\leq k$. Using the Pollard-type inequality mentioned above, we can also get a lower bound for 
\begin{align*}
\int_X 1_{B_{i_1}}(T^{i_1 p} x)1_{B_{i_2}}(T^{i_2p}x)1_{B_{i_3}}(T^{i_3 p}x)\, d\mu(x)
\end{align*} 
for $1\leq i_1<i_2<i_3\leq k$. The question is then, how large $\delta=\min_i \delta_i$ can be under these constraints if \eqref{eq55} fails. This is a combinatorial problem whose solution gives us the value of $c_k$ in Theorem \ref{theo2}.

\subsection{Acknowledgments}

The authors are grateful to the referee a for careful reading of the paper and for useful comments and corrections.

TT was supported by a Simons Investigator grant, the James and Carol Collins Chair, the Mathematical Analysis \& Application Research Fund Endowment, and by NSF grant DMS-1266164.

JT thanks UCLA for excellent working conditions during a visit
there in April 2018, during which a large proportion of this work was completed.

Part of this paper was written while the authors were in residence at MSRI in spring
2017, which is supported by NSF grant DMS-1440140.

\subsection{Notation}\label{sub: not}

We use the following standard arithmetic functions:
\begin{itemize}
\item $\omega(n)$, defined to equal the number of prime factors of $n$ (not counting multiplicity);
\item $\Omega(n)$, defined to equal the number of prime factors of $n$ (counting multiplicity);
\item The \emph{Liouville function} $\lambda(n) = (-1)^{\Omega(n)}$;
\item The \emph{M\"obius function} $\mu(n)$, defined to equal $\lambda(n)$ when $n$ is square-free and $0$ otherwise;
\item The largest prime factor $P^{+}(n)$ of $n$, and the smallest prime factor $P_-(n)$ of $n$ (with the convention $P^{-}(1)=P^{+}(1)=1$);
\item The \emph{Euler totient} function $\varphi(n)$, defined to equal the number $|(\Z/n\Z)^\times|$ of primitive residue classes modulo $n$; and
\item The \emph{von Mangoldt function} $\Lambda(n)$, defined to equal $\log p$ when $n$ is a power $p^j$ of a prime $p$ for some $j \geq 1$, and equal to zero otherwise.
\item The \emph{Dickman function} $\rho(u)$, defined as the unique continuous solution to the delayed differential equation $u\rho'(u)+\rho(u-1)=0$ with the initial condition $\rho(u)=1$ for $0\leq u\leq 1$. As is well-known, we have $\lim_{x\to \infty}\frac{1}{x}|\{n\leq x:\,\, P^{+}(n)\leq x^{u}\}|=\rho(1/u)$; we refer to \cite{ht} for further properties of this function.
\end{itemize}

If $A$ is a finite set, we use $|A|$ to denote its cardinality.  If $A$ is a set of natural numbers, we define the \emph{lower density}
\begin{equation}\label{dminus}
 d_-(A) \coloneqq \liminf_{x \to \infty} \frac{|A \cap [1,x]|}{x},
\end{equation}
the \emph{upper density}
$$ d_+(A) \coloneqq \limsup_{x \to \infty} \frac{|A \cap [1,x]|}{x},$$
and the \emph{asymptotic density}
$$ d(A) \coloneqq \lim_{x \to \infty} \frac{|A \cap [1,x]|}{x}$$
(if it exists).

If $A$ is a finite non-empty set of natural numbers and $f \colon A \to \C$ is a function, we define the average
$$ \E_{n \in A} f(n) \coloneqq \frac{\sum_{n \in A} f(n)}{\sum_{n \in A} 1}$$
and the logarithmic average
$$ \E^{\log}_{n \in A} f(n) \coloneqq \frac{\sum_{n \in A} \frac{f(n)}{n}}{\sum_{n \in A} \frac{1}{n}}.$$
If we average over the variable $p$ instead of $n$, the definitions are same, except that the summation variable is now restricted to be prime.

We utilise the Dickman function $\rho(u)$ that equals to the asymptotic density $d(\{n\in \mathbb{N}:\, P^{+}(n)\leq n^{1/u}\})$; see \cite{ht} for further properties of this function. 

If $A$ is a set, we use $1_A$ to denote the indicator function, thus $1_A(n)=1$ when $n \in A$ and $1_A(n)=0$ otherwise. Similarly, if $E$ is a statement, we let $1_E$ denote the indicator of $E$, thus $1_E=1$ when $E$ is true and $1_E=0$ when $E$ is false.

We use $X \ll Y$, $X \gg Y$, $X = O(Y)$ to denote a bound of the form $|X| \leq CY$ for an absolute constant $C$; if we need to allow $C$ to depend on additional parameters, we denote this by subscripts, thus for instance $X = O_k(Y)$ denotes the bound $|X| \leq C_k Y$ for some $C_k$ depending on $k$.  Given an asymptotic parameter such as $x$ tending to infinity, we use $o(Y)$ to denote a quantity bounded in magnitude by $c(x) Y$ where $c(x)$ goes to zero as $x \to \infty$.

We use $e(x) \coloneqq e^{2\pi i x}$ for the standard character.  We use $n\ \Mod{q}$ for the reduction of $n$ modulo $q$, and $(a_1,\dots,a_k)$ for the greatest common divisor of $a_1,\dots,a_k$.

\section{A correspondence principle}\label{sec: correspondence}

In this section we develop a correspondence principle for weakly stable sets, analogous to the Furstenberg correspondence principle \cite{furst}, which converts problems about establishing patterns in such sets with positive lower density to problems about establishing certain patterns in measure-preserving systems.  The approximately multiplicative structure of weakly stable sets will be incorporated (via the ``entropy decrement argument'' \cite{tao-chowla}) to a certain prime shift in these latter patterns.  This correspondence principle will then be used  in later sections to establish Theorems \ref{theo1}, \ref{theo1.5}, \ref{theo2}.  We remark that the analogous correspondence principle with weakly stable sets replaced by bounded multiplicative functions is essentially contained in the recent work of Frantzikinakis and Host \cite{fhost}.

We first recall the definition of a measure-preserving system.   

\begin{definition}[Measure-preserving systems]
We say that a tuple $(X,\mathcal{X},\mu,T)$ is a \emph{measure-preserving system} if $\mathcal{X}$ is a sigma algebra on $X$, $\mu$ is a measure on $\mathcal{X}$, and $T:X\to X$ is measure-preserving in the sense that $T$ is invertible with $T,T^{-1}$ both measurable with $\mu(T^{-1}A)=\mu(A)$ for all $A\in \mathcal{X}$. We often omit the sigma algebra $\mathcal{X}$ from the notation when it plays no specific role. We further say that $(X,\mathcal{X},\mu,T)$ is a \emph{separable measure-preserving system} if the sigma algebra $\mathcal{X}$ is countably generated.  
\end{definition}

The main result of this section is then as follows.

\begin{theorem}[Correspondence principle for weakly stable sets]\label{weak-stable}
Let $A_1,\dots,A_k \subset \mathbb{N}$ be weakly stable sets.  Suppose that there is a finite index set $I$ and, for each $\alpha \in I$, one has a natural number $m^\alpha$, integers $h_1^\alpha,\dots,h_{m^\alpha}^\alpha$ and indices $c_1^\alpha,\dots,c_{m^\alpha}^\alpha \in \{1,\dots,k\}$ such that
\begin{equation}\label{alop}
 d_-\left( \bigcup_{\alpha \in I} \bigcap_{i=1}^{m^\alpha} (A_{c_i^\alpha} - h_i^\alpha) \right) = 0.
\end{equation}
Then there exists a separable measure preserving system $(X,\mathcal{X},\mu,T)$ and measurable sets $B_1,\dots,B_k \in \mathcal{X}$ such that
\begin{equation}\label{limp}
 \lim_{P \to \infty} \sum_{\alpha \in I} \E^{\log}_{p \leq P} \int_X \prod_{i=1}^{m^\alpha} 1_{B_{c_i^\alpha}}( T^{ph_i^\alpha} x )\ d\mu(x) = 0.
\end{equation}
Furthermore, one can ensure the following additional properties:
\begin{itemize}
\item[(i)]  If for each $j=1,\dots,k$, $A_j$ is uniformly distributed in short intervals with density $\delta_j \in [0,1]$, then for every natural number $q$ and $i=1,\dots,k$ one has
\begin{equation}\label{erg}
 \lim_{H \to \infty} \int_X |\E_{h \leq H} 1_{B_j}( T^{qh} x) - \delta_j|\ d\mu(x) = 0.
\end{equation}
In particular, (by the triangle inequality and shift invariance) each $B_j$ has measure $\delta_j$.
\item[(ii)]  If the $A_j$ are disjoint up to sets of density zero, then the $B_j$ are disjoint up to null sets.
\item[(iii)]  If $d( \bigcup_{j=1}^k A_j ) = 1$, then $\bigcup_{j=1}^k B_j$ has full measure.
\end{itemize}
\end{theorem}

\begin{remark}
For the application to Theorems \ref{theo1} and \ref{theo2}, we are going to to take $I$ to be a singleton and $h_i^{\alpha}=c_i^{\alpha}=i$. For Theorem \ref{theo1.5} in turn, we choose $I=\{1,2,3\}$ and $h_i^{\alpha}=i$, and as $\alpha$ ranges through $I$ the tuples $(c_i^{\alpha})_{i\leq 3}$ run through solutions to $c_1^{\alpha}+c_2^{\alpha}+c_3^{\alpha}=c \Mod{3}$. 
\end{remark}

We remark that by the ergodic theorem, the conclusion \eqref{erg} is equivalent to $1_{B_i} - \delta_i$ being orthogonal to the \emph{profinite factor} of $X$, defined as the factor generated by all the periodic functions on $X$ (that is, functions $f:X\to \mathbb{C}$ with $f(T^kx)=f(x)$ for some natural number $k$ and almost all $x\in X$).  The presence of the dilation factor $p$ in the shifts $T^{ph_i^\alpha}$ in \eqref{limp} is a key feature of this principle that is not present in the classical Furstenberg correspondence principle, and is introduced via the entropy decrement argument from \cite{tao-chowla}.  We remark that the existence of the limit in \eqref{limp} can also be derived from the general convergence results for multiple ergodic averages along the primes in \cite{fhk}, \cite{wz}, and the logarithmically averaged limit $\E^{\log}_{p \leq P}$ can then be replaced by the ordinary average $\E_{p \leq P}$.  In fact we have a useful formula for the limit; see Proposition \ref{limf} below.

We now prove the theorem.  Let 
$$ S \coloneqq \bigcup_{\alpha \in I} \bigcap_{i=1}^{m^\alpha} (A_{c_i^\alpha} - h_i^\alpha) $$
denote the set in \eqref{alop}.  By hypothesis, we have $d_-(S)=0$, thus we can find a sequence $x_l$ tending to infinity such that
$$ \E_{n \leq x_l} 1_S(n) = o(1)$$
as $l \to \infty$.
In particular, if $1 \leq \omega_l \leq x_l$ goes to infinity sufficiently slowly, one has
$$ \omega_l \E_{n \leq x_l} 1_S(n) = o(1)$$
which implies in particular that
$$ \E_{x_l/\omega_l \leq n \leq x_l}^{\log} 1_S(n) = o(1).$$
Since 
$$ 1_S(n) = \sum_{\alpha \in I} \prod_{i=1}^{m^\alpha} 1_{A_{c_i^\alpha}}(n+h_i^\alpha)$$
we thus have
\begin{equation}\label{xll}
 \E_{x_l/\omega_l \leq n \leq x_l}^{\log} \prod_{i=1}^{m^\alpha} 1_{A_{c_i^\alpha}}(n+h_i^\alpha) = o(1)
\end{equation}
for each $\alpha \in I$.

For each $j=1,\dots,k$, the set $A_j$ is weakly stable by hypothesis.  Let $B_{x,j}$ be the sets corresponding to $A_j$ as per Definition \ref{def_weakstable}.  Then for each prime $p$, one has
$$ \E_{n \leq x_l} 1_{p \nmid n} |1_{A_j}(n) - 1_{B_{x_l,j}}(pn)| = o(1)$$
as $l \to \infty$, which for $\omega_l$ sufficiently slowly growing depending on $p$ implies that
\begin{equation}\label{exi}
 \E_{x_l/\omega_l \leq n \leq x_l}^{\log} 1_{p \nmid n} |1_{A_j}(n) - 1_{B_{x_l,j}}(pn)| = o(1).
\end{equation}
By a diagonalisation argument, one can select $\omega_l$ so that \eqref{exi} holds for \emph{all} primes $p$ (of course, the decay rate will almost certainly not be uniform in $p$).

Restoring the case $p|n$, we have
$$ \E_{x_l/\omega_l \leq n \leq x_l}^{\log} |1_{A_j}(n) - 1_{B_{x_l,j}}(pn)| \ll \frac{1}{p} + o(1) ,$$
and hence also
$$ \E_{x_l/\omega_l \leq n \leq x_l}^{\log} |1_{A_{c_i^\alpha}}(n + h_i^\alpha) - 1_{B_{x_l,c_i^\alpha}}(pn + p h_i^\alpha)| \ll \frac{1}{p} + o(1) $$
for all $\alpha \in I$ and $i=1,\dots,m^\alpha$.  From this, \eqref{xll}, and the triangle inequality we conclude that
$$ \E_{x_l/\omega_l \leq n \leq x_l}^{\log} \prod_{i=1}^{m^\alpha} 1_{B_{x_l,c_i^\alpha}}(pn+ph_i^\alpha) \ll \frac{1}{p} + o(1)$$
for all primes $p$, all $\alpha \in I$, and $i=1,\dots,m^\alpha$, where we allow implied constants in the asymptotic notation to depend on $I$ and the $m^\alpha$.  Writing this average in terms of $pn$ instead of $n$ (which only impacts the logarithmic average in $n$ by a negligible amount, other than by now restricting $n$ to multiples of $p$), we obtain
$$ \E_{x_l/\omega_l \leq n \leq x_l}^{\log} \prod_{i=1}^{m^\alpha} 1_{B_{x_l,c_i^\alpha}}(n+ph_i^\alpha) p 1_{p|n} \ll \frac{1}{p} + o(1).$$

If we logarithmically average over primes $p \leq P$, we conclude from the convergence of $\sum_p \frac{1}{p^2}$ and the divergence of $\sum_p \frac{1}{p}$ that
$$\lim_{P \to \infty} \limsup_{l \to \infty} \E_{p \leq P}^{\log} \E_{x_l/\omega_l \leq n \leq x_l}^{\log} \prod_{i=1}^{m^\alpha} 1_{B_{x_l,c_i^\alpha}}(n+ph_i^\alpha) p 1_{p|n} = 0.$$
On the other hand, by the entropy decrement argument \cite[Theorem 3.6]{tt-elliott} we have
$$\lim_{P \to \infty} \limsup_{l \to \infty} \E_{p \leq P}^{\log} \E_{x_l/\omega_l \leq n \leq x_l}^{\log} \prod_{i=1}^{m^\alpha} 1_{B_{x_l,c_i^\alpha}}(n+ph_i^\alpha) (p 1_{p|n} - 1) = 0.$$
We conclude from the triangle inequality that
$$\lim_{P \to \infty} \limsup_{l \to \infty} \E_{p \leq P}^{\log} \E_{x_l/\omega_l \leq n \leq x_l}^{\log} \prod_{i=1}^{m^\alpha} 1_{B_{x_l,c_i^\alpha}}(n+ph_i^\alpha) = 0$$
for all $\alpha \in I$.

Next, let $\plim: \ell^\infty(\N) \to \C$ denote a generalised limit functional, that is to say
a bounded linear functional on $\ell^\infty(\N)$ that extends the limit functional on convergent sequences, and such that
$$ \liminf_{l \to\infty} a_l \leq \plim (a_l)_{l \in \N} \leq \limsup_{n \to\infty} a_l$$
for all bounded real-valued sequences $a_n$.  The existence of such a generalised limit functional easily follows from the Hahn-Banach theorem (or from the existence of non-principal ultrafilters on $\N$).  Then we have
\begin{equation}\label{limpi}
\lim_{P \to \infty} \plim \left( \E_{p \leq P}^{\log} \E_{x_l/\omega_l \leq n \leq x_l}^{\log} \prod_{i=1}^{m^\alpha} 1_{B_{x_l,c_i^\alpha}}(n+ph_i^\alpha)\right)_{l \in \N} = 0.
\end{equation}

Let $X$ denote the product space $(\{0,1\}^k)^\Z$ of sequences $(x_{c,m})_{c \in \{1,\dots,k\}, m \in \Z}$ of numbers $x_{c,m} \in \{0,1\}$ with the product sigma algebra ${\mathcal X}$ (so in particular, $X$ is a compact Hausdorff space with separable sigma algebra ${\mathcal X}$) and the shift 
$$T(x_{c,m})_{c \in \{1,\dots,k\}, m \in \Z} \coloneqq (x_{c,m+1})_{c \in \{1,\dots,k\}, m \in \Z}.$$  
We define a probability measure $\mu$ on $X$ by requiring that
\begin{align}\label{xphi}
\begin{aligned}
 \int_X \prod_{\beta \in J} 1_{x_{c_\beta,m_\beta} = 1}\ d\mu(x) = \plim \left( \mathbb{E}_{x_l/\omega_l \leq n \leq x_l}^{\log} \prod_{\beta \in J} 1_{B_{x_l,c_\beta}}(n + m_\beta) \right)_{l \in \N}
\end{aligned}
\end{align}
for any finite index set $J$, any $c_\beta \in \{1,\dots,k\}$, and any integers $m_\beta$.  The existence (and uniqueness) of this measure follows from the Kolmogorov extension theorem.  The measure $\mu$ is a probability measure that is invariant under the shift $T$, since the right-hand side of  \eqref{xphi} remains invariant when the $m_\beta$ are replaced by $m_\beta+1$. Next, we define the measurable sets $B_j$ for $j=1,\dots,k$ by the formula
$$ B_j \coloneqq \{ (x_{c,m})_{c \in \{1,\dots,k\}, m \in \Z} \in X: x_{j,0} = 1 \},$$
then one can rewrite the left-hand side of \eqref{xphi} as
\begin{equation}\label{xphi-2}
 \int_X \prod_{\beta \in J} 1_{B_{c_{\beta}}}(T^{m_\beta} x)\ d\mu(x).
\end{equation}
In particular, from \eqref{limpi} one has
$$
\lim_{P \to \infty}\mathbb{E}_{p\leq P}^{\log} \int_X \prod_{i=1}^{m^\alpha} 1_{B_{c_i^\alpha}}(T^{ph_i^\alpha} x)\ d\mu = 0$$
for all $\alpha \in I$, which gives \eqref{limp}.

Now we prove (ii).  If $A_j, A_{j'}$ are disjoint up to zero density sets, then 
$$
 \E_{x_l/\omega_l \leq n \leq x_l}^{\log} 1_{A_j}(n) 1_{A_{j'}(n)} = o(1).$$
Repeating the previous arguments using this bound in place of \eqref{xll}, we eventually arrive at
$$ \int_X 1_{B_j}(x) 1_{B_{j'}}(x)\ d\mu(x) = 0,$$
and hence $B_j, B_{j'}$ are disjoint up to null sets.  This gives (ii).  Similarly, if $\bigcup_{j=1}^k A_j$ has density one, then
$$
 \E_{x_l/\omega_l \leq n \leq x_l}^{\log} \prod_{j=1}^k (1 - 1_{A_j}(n)) = o(1),$$
and then by repeating the previous arguments
$$ \int_X \prod_{j=1}^k (1-1_{B_j}(x))\ d\mu(x) = 0,$$
so that $\bigcup_{j=1}^k B_j$ has full measure.  This establishes (iii).

Now we turn to (i).  Fix $b,q,j$, let $\eps>0$, let $Q$ be sufficiently large (depending on $b,q,\eps$), and then let $H$ be sufficiently large (depending on $b,q,\eps,Q$).  Further, let $p$ be a prime in $[\log Q, Q]$. Since $A_j$ is uniformly distributed in short intervals with density $\delta_j$, we then conclude (if $\omega_l$ grows slowly enough) that 
\begin{equation}\label{abd}
\sup_{p\in [\log Q,Q]} \mathbb{E}_{x_l/\omega_l\leq y\leq x_l}^{\log}||A_j\cap [y/p,y/p+qH/p]\cap (q\mathbb{Z}+b\overline{p})|-\delta_j H/p| =o(1),
\end{equation}
where $\overline{p}$ denotes the inverse of $p$ in $\mathbb{Z}/q\mathbb{Z}$ (this exists since $p \geq \log Q > q$ for $Q$ large enough).  Also, since $A_j$ is weakly stable, we have
$$ \E_{n \leq x/p: p \nmid n} |1_{A_j}(n) - 1_{B_{x,j}}(pn)| = o(1),$$
and hence 
\begin{equation}\label{bbd}
 \sup_{p\in [\log Q,Q]} \mathbb{E}_{x_l/\omega_l\leq y\leq x_l}^{\log} \E_{y/p \leq n \leq y/p + qH/p: p \nmid n} |1_{A_j}(n) - 1_{B_{x,j}}(pn)| = o(1).
\end{equation}
From \eqref{bbd} we have in particular that for each $p \in [\log Q,Q]$
$$  \mathbb{E}_{x_l/\omega_l\leq y\leq x_l}^{\log}\E_{n \in [y/p,y/p+qH/p] \cap (q\mathbb{Z}+b\overline{p}): p \nmid n} |1_{A_j}(n) - 1_{B_{x_l,j}}(pn)| = o(1),$$
and hence on removing the $p \nmid n$ constraint
$$  \mathbb{E}_{x_l/\omega_l\leq y\leq x_l}^{\log}\E_{n \in [y/p,y/p+qH/p] \cap (q\mathbb{Z}+b\overline{p})} |1_{A_j}(n) - 1_{B_{x_l,j}}(pn)| \ll \frac{1}{p} + o(1).$$
Meanwhile, from \eqref{abd} one has
$$  \mathbb{E}_{x_l/\omega_l\leq y\leq x_l}^{\log} |\E_{n \in [y/p,y/p+qH/p] \cap (q\mathbb{Z}+b\overline{p})} (1_{A_j}(n)-\delta_j)|=o(1).$$
By the triangle inequality, we conclude that
$$  \mathbb{E}_{x_l/\omega_l\leq y\leq x_l}^{\log} |\E_{n \in [y/p,y/p+qH/p] \cap (q\mathbb{Z}+b\overline{p})} (1_{B_{x_l,j}}(pn) - \delta_j)| \ll 1/p+o(1),$$
or equivalently
$$  \mathbb{E}_{x_l/\omega_l\leq y\leq x_l}^{\log}|\E_{n \in [y,y+qH] \cap (q\mathbb{Z}+b)} (1_{B_{x_l,j}}(n) - \delta_j) 1_{p|n}| \ll \frac{1}{p}\left(\frac{1}{p} + o(1)\right).$$
We can estimate $\frac{1}{p}+o(1)$ by $O(\eps)$ for $Q$ sufficiently large.  We then sum in $p$ and use the triangle inequality to conclude that
$$  \mathbb{E}_{x_l/\omega_l\leq y\leq x_l}^{\log}\big|\E_{n \in [y,y+qH] \cap (q\mathbb{Z}+b)} (1_{B_{x_l,j}}(n) - \delta_j) \sum_{\log Q \leq p \leq Q} 1_{p|n}\big| \ll \eps \log \log Q$$
for $Q$ sufficiently large. On the other hand, from the Turan-Kubilius inequality (or a direct second moment calculation) we have
$$ \E_{n \in [y,y+qH] \cap (q\mathbb{Z}+b)} \left|\sum_{\log Q \leq p \leq Q} 1_{p|n} - \log\log Q\right|^2 \ll \eps^2 (\log\log Q)^2,$$
and hence by Cauchy--Schwarz
$$ \E_{n \in [y,y+qH] \cap (q\mathbb{Z}+b)} |1_{B_{x_l,j}}(n) - \delta_j| \left|\sum_{\log Q \leq p \leq Q} 1_{p|n} - \log\log Q\right| \ll \eps \log \log Q.$$
From the triangle inequality, we thus have
$$  \mathbb{E}_{x_l/\omega_l\leq y\leq x_l}^{\log}|\E_{n \in [y,y+qH] \cap (q\mathbb{Z}+b)} (1_{B_{x_l,j}}(n) -\delta_j) |\ll \eps.$$
This implies that
$$ \limsup_{l \to \infty} \E^{\log}_{x_l/\omega_l \leq n \leq x_l: n = b\ (q)} | \E_{h \leq H} (1_{B_{x_l,j}}(n + qh) -\delta_j)| \ll \eps;$$
averaging in $b$, this implies
$$ \limsup_{l \to \infty} \E^{\log}_{x_l/\omega_l \leq n \leq x_l} | \E_{h \leq H} (1_{B_{x_l,j}}(n + qh) -\delta_j)| \ll \eps,$$
and thus
$$ \lim_{H \to \infty} \limsup_{l \to \infty} \E^{\log}_{x_l/\omega_l \leq n \leq x_l} | \E_{h \leq H} (1_{B_{x_l,j}}(n + qh) -\delta_j)|^2 = 0.$$
Using \eqref{xphi}, \eqref{xphi-2} and expanding the square, we conclude that
$$ \lim_{H \to \infty} \int_X | \E_{h \leq H} (1_{B_{j}}(T^{qh} x) -\delta_j)|^2\ d\mu(x) = 0,$$
and \eqref{erg} follows from the Cauchy--Schwarz inequality.  This completes the proof of Theorem \ref{weak-stable}.

In view of this correspondence principle (taken in the contrapositive), Theorems \ref{theo1}, \ref{theo1.5}, \ref{theo2} are immediate consequences of the following ergodic-theoretic counterparts (specialised to the case when $F_j = 1_{B_j}$ are indicator functions).

\begin{theorem}[Main theorem, ergodic version]\label{theo-erg} Let $F_1,\dots,F_k: X \to [0,1]$ be measurable functions on a measure-preserving system $(X,\mathcal{X},\mu,T)$, and let $\delta_1,\dots,\delta_k \in (0,1]$ be such that
\begin{align}\label{eq51} 
\lim_{H \to \infty} \int_X |\E_{h \leq H} F_j( T^{qh} x) - \delta_j|\ d\mu(x) = 0 
\end{align}
for all $q \geq 1$ and $j=1,\dots,k$.
\begin{itemize}
\item[(i)]  ($k=3$, large density) If $k=3$ and $\delta_1+\delta_2+\delta_3 > 1$, then
$$
\limsup_{P \to \infty} \E^{\log}_{p \leq P} \int_X F_1( T^p x) F_2( T^{2p} x) F_3( T^{3p} x)\ d\mu(x) > 0.
$$
\item[(ii)]  ($k=3$, critical density, first part)  If $k=3$ and $\delta_1+\delta_2+\delta_3 = 1$, then
$$
\limsup_{P \to \infty} \sum_{\substack{c_1,c_2,c_3\in \{0,1,2\}\\c_1+c_2+c_3\equiv c \Mod{3}}} \E^{\log}_{p \leq P} \int_X F_{c_1}( T^p x) F_{c_2}( T^{2p} x) F_{c_3}( T^{3p} x)\ d\mu(x) > 0
$$
for all $c=0,1,2$.
\item[(iii)]  ($k=3$, critical density, second part)  If $k=3$, $\delta_1+\delta_2+\delta_3 = 1$, $\delta_1 \neq \delta_3$, and $F_1+F_2+F_3=1$ almost everywhere, then
$$
\limsup_{P \to \infty} \E^{\log}_{p \leq P} \int_X F_1( T^p x) F_2( T^{2p} x) F_3( T^{3p} x)\ d\mu(x) > 0.
$$
\item[(iv)] ($k>3$)  If $k>3$ and $\delta_1,\dots,\delta_k > c_k$ (where $c_k$ is as in Theorem \ref{theo2}), then
$$
\limsup_{P \to \infty} \E^{\log}_{p \leq P} \int_X F_1( T^p x) \dots F_k( T^{kp} x)\ d\mu(x) > 0.
$$
\end{itemize}
\end{theorem}

\begin{remark}\label{rmk2}
The condition $\delta_1\neq \delta_3$ in part (iii) is necessary. To see this, let $X=(\mathbb{R}/\mathbb{Z})\times (\mathbb{Z}/2\mathbb{Z})$, equipped with its Haar measure and the measure-preserving map $T(x,n)=(x+\alpha,n+1)$ for $\alpha$ irrational. In addition, for $0<\delta_2<1$ define the intervals
\begin{align*}
I_1=[\frac{\delta_2}{2},\frac{1}{2}],\quad I_2=[0,\frac{\delta_2}{2})\cup [\frac{1}{2},\frac{1+\delta_2}{2}],\quad I_3=[\frac{1+\delta_2}{2},1)
\end{align*}
and the functions $F_i(x,n)=1_{I_i}(x+(n\%2)/4)$, where $n\%2$ equals $0$ when $n$ is even and $1$ when $n$ is odd. We then have $F_1+F_2+F_3\equiv 1$. By Weyl's equidistribution theorem, condition \eqref{eq51} is satisfied for $j=1,2,3$ with densities $(1-\delta_2)/2,\delta_2,(1-\delta_2)/2$, respectively. However, for any $p$, we have $F_1(T^p x)F_2(T^{2p}x)F_3(T^{3p}x)=0$, since for any $x,y\in \mathbb{R}/\mathbb{Z}$ we cannot simultaneously have $x+y\in I_1$, $x+2y\in I_2\pm 1/4$, $x+3y\in I_3$.\\
Analogously, if we define the sets of \emph{integers}
\begin{align*}
A_i=\{n\equiv 0\Mod{2}:\,\alpha n\in I_i \mod 1\}\cup \{n\equiv 1\Mod{2}:\,\alpha n-1/4\in I_i \mod 1\}
\end{align*}
for $i=1,2,3$, then $A_1,A_2,A_3$ are uniformly distributed in short intervals with densities $(1-\delta_2)/2,\delta_2,(1-\delta_2)/2$, respectively, but $n+d\in A_1,n+2d\in A_2,n+3d\in A_3$ for $d$ odd never happens. 
\end{remark}

To prove this theorem, we will use the following explicit formula for the limit of multiple ergodic averages along primes, which is essentially implicit in \cite{fhk}.

\begin{proposition}[Limit formula]\label{limf}  Let $F_1,\dots,F_k \in L^\infty(X)$ be bounded measurable functions on a measure-preserving system $(X,\mathcal{X},\mu,T)$.  Then
$$
\lim_{P \to \infty} \E^{\log}_{p \leq P} \int_X F_1( T^p x) \dots F_k( T^{kp} x)\ d\mu(x) 
= \lim_{w \to \infty} \lim_{P \to \infty} \E^{\log}_{d \leq P: (d,W)=1} \int_X F_1( T^{d} x) \dots F_k( T^{kd} x)\ d\mu(x)
$$
where $W \coloneqq \prod_{p \leq w} p$.
\end{proposition}

We remark that the convergence of the inner limit on the right-hand side was first established by Host and Kra \cite{hk}; see also \cite{ziegler} for an alternate proof.

\begin{proof}  To abbreviate notation we write $A(d) \coloneqq \int_X F_1(T^d x) \dots F_k(T^{kd} x)\ d\mu(x)$.  It suffices to show that
$$
\lim_{w \to \infty} \limsup_{P \to \infty} |\E^{\log}_{p \leq P} A(p) - \E^{\log}_{d \leq P: (d,W)=1} A(d)| = 0.$$
By summation by parts it will suffice to show that
$$
\lim_{w \to \infty} \limsup_{P \to \infty} |\E_{p \leq P} A(p) - \E_{d \leq P: (d,W)=1} A(d)| = 0,$$
and by dyadic decomposition it then suffices to show that
$$
\lim_{w \to \infty} \limsup_{P \to \infty} |\E_{P \leq p \leq 2P} A(p) - \E_{P \leq d \leq 2P: (d,W)=1} A(d)| = 0.$$
Equivalently, we need to show that
$$ \E_{P \leq p \leq 2P} A(p)= \E_{P \leq d \leq 2P: (d,W)=1} A(d) + o(1)$$
as $P \to \infty$, if $w = w(P)$ goes to infinity sufficiently slowly as $P \to \infty$.  By splitting into residue classes modulo $W$, it suffices to show that
$$ \E_{P \leq p \leq 2P: p = b\ \Mod{W}} A(p) = \E_{P \leq d \leq 2P: d = b\ \Mod{W}} A(d) + o(1)$$
uniformly for all $1 \leq b < W$ coprime to $W$.

Using the von Mangoldt function $\Lambda$ and the prime number theorem in arithmetic progressions, we can write the left-hand side as
$$ \E_{P \leq d \leq 2P: d = b\ \Mod{W}} \frac{\phi(W)}{W} \Lambda(d) A(d),$$
so it suffices to show that
$$ \E_{P \leq d \leq 2P: d = b\ \Mod{W}} (\frac{\phi(W)}{W} \Lambda(d)-1) A(d) = o(1),$$
or equivalently that
$$ \E_{P/W \leq d \leq 2P/W} (\Lambda_{b,W}(d)-1) \int_X F_1( T^{Wd+b} x) \dots F_k( T^{Wkd+kb} x)\ d\mu(x) = o(1),$$
where $\Lambda_{b,W}(d) \coloneqq \frac{\phi(W)}{W} \Lambda(Wd+b)$.  Replacing $x$ by $T^n x$ for $n \leq P$ and averaging, it suffices to show that
$$ \E_{P/W \leq d \leq 2P/W} \E_{n \leq P} \int_X (\Lambda_{b,W}(d)-1) F_1( T^{n+Wd+b} x) \dots F_k( T^{n+Wkd+kb} x)\, d\mu(x) = o(1)$$
uniformly in $b$.  By the generalised von Neumann theorem in the form of \cite[Lemma 5.2]{tt-chowla}, this will follow from the claim
$$ \| \Lambda_{b,W}(d)-1\|_{U^k[2P/W]} = o(1)$$
where the Gowers norm $U^k$ is defined for instance in \cite{gt-linear}.  But this follows from \cite[Theorem 7.2]{gt-linear} (combined with the main results of \cite{gt-mobius}, \cite{green_tao_ziegler}).
\end{proof}

It will thus suffice to prove the following slightly stronger version of Theorem \ref{theo-erg}.

\begin{theorem}[Main theorem, ergodic version, II]\label{theo-erg2} Let $F_1,\dots,F_k: X \to [0,1]$ be measurable functions on a measure-preserving system $(X,\mathcal{X},\mu,T)$, and let $\delta_1,\dots,\delta_k \in (0,1]$ be such that
\begin{equation}\label{erg2} \lim_{H \to \infty} \int_X |\E_{h \leq H} F_j( T^{qh} x) - \delta_j|\ d\mu(x) = 0
\end{equation}
for all $q \geq 1$ and $j=1,\dots,k$.  We allow implied constants to depend on $k,\delta_1,\dots,\delta_k$.  Let $W$ be a natural number.
\begin{itemize}
\item[(i)]  ($k=3$, large density) If $k=3$ and $\delta_1+\delta_2+\delta_3 > 1$, then
$$
\limsup_{P \to \infty} \E_{d \leq P: (d,W)=1} \int_X F_1( T^d x) F_2( T^{2d} x) F_3( T^{3d} x)\ d\mu(x) \gg 1.
$$
\item[(ii)]  ($k=3$, critical density, first part)  If $k=3$ and $\delta_1+\delta_2+\delta_3 = 1$, then
\begin{equation}\label{limps}
\limsup_{P \to \infty} \sum_{\substack{c_1,c_2,c_3\in \{0,1,2\}\\c_1+c_2+c_3\equiv c \Mod{3}}} \E_{d \leq P: (d,W)=1} \int_X F_{c_1}( T^d x) F_{c_2}( T^{2d} x) F_{c_3}( T^{3d} x)\ d\mu(x) \gg 1
\end{equation}
for all $c=0,1,2$.
\item[(iii)]  ($k=3$, critical density, second part)  If $k=3$, $\delta_1+\delta_2+\delta_3 = 1$, $\delta_1 \neq \delta_3$, and $F_1+F_2+F_3=1$ almost everywhere, then
$$
\limsup_{P \to \infty} \E_{d \leq P: (d,W)=1} \int_X F_1( T^d x) F_2( T^{2d} x) F_3( T^{3d} x)\ d\mu(x) \gg 1.
$$
\item[(iv)] ($k>3$)  If $k>3$ and $\delta_1,\dots,\delta_k > c_k$ (where $c_k$ is as in Theorem \ref{theo2}), then
\begin{equation}\label{dw}
\limsup_{P \to \infty} \E_{d \leq P: (d,W)=1} \int_X F_1( T^d x) \dots F_k( T^{kd} x)\ d\mu(x) \gg 1.
\end{equation}

\end{itemize}
\end{theorem}

A key point here is that the lower bound is independent of $W$ (and of the system $X$).

\section[The main theorems for k=3]{The main theorems for $k=3$} \label{sec: 3}

In this section we prove parts (i)-(iii) of Theorem \ref{theo-erg2}.  We begin with some standard reductions.

\subsection[Reduction to the case of X being ergodic]{Reduction to the case of $X$ being ergodic}

We claim that to prove any part of Theorem \ref{theo-erg2}, it suffices to do so in the case when the measure-preserving system $X$ is ergodic (that is to say, all $T$-invariant subsets of $X$ have measure zero or full measure).  For sake of discussion we only present this in the case (ii), as the other cases are similar.  Let $X$ be a separable measure-preserving system that is not necessarily ergodic. Applying the ergodic decomposition (see e.g. \cite[Theorem 3.42]{glasner}) one can obtain a disintegration
\begin{align}\label{eqq3}
\mu = \int_Y \mu_y\ d\nu(y),
\end{align} 
where $(Y,\nu)$ is the $T$-invariant factor of $(X,\mu)$, and for $\nu$-almost every $y$, the $(X,T,\mu_y)$ are ergodic measure-preserving systems.  Assume that Theorem \ref{theo-erg2}(ii) is established whenever $X$ is ergodic.  By dominated convergence, \eqref{eqq3} and \eqref{erg2}, we have
\begin{align*}
&\int_{Y}\lim_{H\to \infty}\int_{X}|\E_{h\in [H]}F_c(T^{qh}x)-\delta_c|d\mu_y(x)d\nu(y)\\
&=\lim_{H\to \infty} \int_{Y}\int_{X}|\E_{h\in [H]}F_c(T^{qh}x)-\delta_c|d\mu_y(x)d\nu(y)=0.   
\end{align*}
Thus, for any $c = 1,2,3$, $q \geq 1$ and $\nu$-almost every $y$, we have that $\E_{h \in [H]} F_c(T^{qh} x)$ converges in $L^1(X,\mu_y)$ norm to $\delta_c$ as $H \to \infty$.  Applying Theorem \ref{theo-erg2}(ii) in the ergodic case, we conclude that for every $W$ and $c \in \Z/3\Z$, one has
$$ \liminf_{P \to \infty} \sum_{\substack{c_1,c_2,c_3 \in \{1,2,3\}\\ c_1+c_2+c_3 = c\ \Mod{3}}} \E_{P \leq r \leq 2P: (r,W)=1} \int_X F_{c_1}(T^r x) F_{c_2}(T^{2r} x) F_{c_3}(T^{3r} x)\ d\mu_y(x) \gg 1 $$
for $\nu$-almost every $y$.  Integrating in $y$ and applying\footnote{Though it is not strictly necessary, one could use the results of \cite{furst} (see also \cite{hk}) to upgrade the limit inferior here to a limit.} Fatou's lemma, this implies that
$$
 \liminf_{P \to \infty} \sum_{\substack{c_1,c_2,c_3 \in \{1,2,3\}\\ c_1+c_2+c_3 = c\ \Mod{3}}} \E_{P \leq r \leq 2P: (r,W)=1} \int_X F_{c_1}(T^r x) F_{c_2}(T^{2r} x) F_{c_3}(T^{3r} x)\ d\mu(x) \gg 1,
$$
giving Theorem \ref{theo-erg2}(ii) in the general case.  A similar argument works for all other components of Theorem \ref{theo-erg2}.

\subsection[Reduction to the case of X being a Kronecker system]{Reduction to the case of $X$ being a Kronecker system}

Next we make a reduction of parts (i)-(iii) of Theorem \ref{theo-erg2} to the case when $X$ is a \emph{Kronecker system}, by which we mean that $X$ is a compact separable abelian group with shift $T$ given by a translation $T: x \mapsto x+\alpha$; the argument here relies crucially on the fact that $k=3$, and does not extend to part (iv).  Again, we only detail this reduction for the case (ii).  If $(X,T,\mu)$ is an ergodic separable measure-preserving system, then (as is well known, see e.g. \cite{furst-weiss}) we can form the \emph{Kronecker factor} $(Z^1, S, \nu)$, which is a Kronecker system together with a factor map $\pi: X \mapsto Z^1$ that pushes forward $\mu$ to $\nu$ and intertwines $T$ and $S$ (with the measurable functions on $Z^1$ pulling back to the functions on $X$ generated by the eigenfunctions of $T$).    Furthermore, any average of the form
$$ \lim_{P \to \infty} \E_{r \in [P]} \int_X G_1(T^{ar} x) G_2(T^{br} x) G_3(T^{cr} x)\ d\mu$$
for distinct integers $a,b,c$ will vanish whenever at least one of the functions $G_1,G_2,G_3 \in L^\infty(X)$ is orthogonal to the Kronecker factor in the sense that the conditional expectation $\E(G_i|Z^1)$ vanishes for some $i$.  As such, we see (as in \cite{furst-weiss}) that the Kronecker factor is \emph{characteristic} for the average in \eqref{limp}, in the sense that one can replace each of the functions $F_c$ by the conditional expectation $\E(F_c|Z^1)$ without affecting the average.  The Kronecker factor is also characteristic for the ergodic averages in \eqref{erg2}.  Finally, as the functions $F_1,F_2,F_3$ take values in $[0,1]$ and sum to $1$, the same is true for $\E(F_1|Z^1), \E(F_2|Z^1), \E(F_3|Z^1)$. As such, we see that to prove Theorem \ref{theo-erg2}(ii) for the functions $F_1,F_2,F_3$ it suffices to do so for $\E(F_1|Z^1)$, $\E(F_2|Z^1)$, $\E(F_3|Z^1)$.  Thus Theorem \ref{theo-erg2}(ii) for general ergodic systems will follow from the case of Kronecker systems. Similarly for parts (i) or (iii) of this theorem.

\subsection[Reduction to the case of X being a Kronecker system corresponding to a Lie group]{Reduction to the case of $X$ being a Kronecker system corresponding to a Lie group}

We make a further reduction of parts (i)-(iii) of Theorem \ref{theo-erg2} to the case when the Kronecker system is a compact abelian \emph{Lie} group.  Again, we only discuss the case (ii).  It is easy to see that a general Kronecker system $X$ is expressible as the inverse limit of Kronecker systems $X_n$ that are compact abelian Lie groups (see also \cite{hk} for the generalisation of this claim to higher step).  Suppose that Theorem \ref{theo-erg2}(ii) has been proven for Kronecker systems that are compact abelian Lie groups.  If $F_1,F_2,F_3, X$ are as in that theorem, then $\E(F_c|X_n)$ will converge in $L^1(X,\mu)$ norm to $F_c$ for $c \in \Z/3\Z$.  Applying conditional expectations to \eqref{erg2} and using the dominated convergence theorem, we see that this hypothesis continues to hold if each function $F_c$ is replaced with $\E(F_c|Z_n)$.  Thus, by hypothesis, we see that for any $c \in \Z/3\Z$, we have
\begin{align*} 
\liminf_{P \to \infty} \sum_{\substack{c_1,c_2,c_3 \in \{1,2,3\}\\ c_1+c_2+c_3 = c\ \Mod{3}}} \E_{P \leq r \leq 2P: (r,W)=1}\int_X \E(F_{c_1}|Z_n)(T^r x) \E(F_{c_2}|Z_n)(T^{2r} x)\cdot\\
\cdot \E(F_{c_3}|Z_n)(T^{3r} x)\ d\mu(x) \gg 1,
\end{align*}
with the implied constants uniform in $n$.  Taking limits in $n$, we obtain Theorem \ref{theo-erg2}(ii) for arbitrary Kronecker systems.  Similarly for Theorem \ref{theo-erg2}(i) or Theorem \ref{theo-erg2}(iii).

\subsection{Main argument}

We continue the proof of Theorem \ref{theo-erg2}(ii).
Henceforth $X$ is a Kronecker system that is a compact abelian Lie group.  As the translation map $T$ is ergodic, the system $X$ must (up to isomorphism) take the form $X = G \times \Z/M\Z$ for some \emph{connected} compact abelian Lie group (i.e. a torus) and some $M \geq 1$, with shift given by $T(x,a) \coloneqq (x+\alpha,a+1)$ for some $\alpha \in G$, such that the translation $x \mapsto x+\alpha$ is ergodic on $G$ (and hence totally ergodic, since $G$ is connected and so the Pontragyin dual $\hat G$ is torsion-free).  Applying the hypothesis \eqref{erg2} with $q=M$, we conclude in particular that
$$
 \lim_{H \to \infty} \int_G | \E_{h \in [H]} F_c(x+M\alpha h, a) - \delta_c|\ d\mu_G(x) = 0$$
for all $c = 1,2,3$ and $a \in \Z/M\Z$, where $\mu_G$ is the Haar probability measure on $G$. By the ergodic theorem and total ergodicity of the shift $x \mapsto x+\alpha$, we thus have
\begin{equation}\label{fmean}
 \int_G F_c(x,a)\ d\mu_G(x) = \delta_c
\end{equation}
for all $c=1,2,3$ and $a \in \Z/M\Z$.

Next, we expand the left-hand side of \eqref{limps} as
\begin{align*}
\liminf_{P \to \infty} \sum_{\substack{c_1,c_2,c_3 \in \{1,2,3\}\\ c_1+c_2+c_3 = c\ \Mod{3}}} \E_{P \leq r \leq 2P: (r,W)=1} \E_{a \in \Z/M\Z} \int_G F_{c_1}(x+r\alpha,a+r) F_{c_2}(x+2r\alpha,a+2r)\cdot\\
\cdot F_{c_3}(x+3r\alpha, a+3r)\ d\mu_G(x).
\end{align*} 
We split $r$ into residue classes modulo $MW$ to write this as
\begin{align}\begin{split}\label{eq50}
\liminf_{P \to \infty} \sum_{\substack{c_1,c_2,c_3 \in \{1,2,3\}\\ c_1+c_2+c_3 = c\ \Mod{3}}} \E_{b \in [MW]: (b,W)=1} \E_{P \leq r \leq 2P: r = b\ \Mod{MW}} \E_{a \in \Z/M\Z} \int_G F_{c_1}(x+r\alpha,a+b)\cdot\\
\cdot F_{c_2}(x+2r\alpha,a+2b) F_{c_3}(x+3r\alpha, a+3b)\ d\mu_G(x).
\end{split}
\end{align} 
A standard calculation (see \cite[Theorem 2.1]{furst-weiss}) shows that
\begin{align*}
&\lim_{P \to \infty} \E_{P \leq r \leq 2P: r = b\ (MW)} \int_G F_{c_1}(x+r\alpha,a+b) F_{c_2}(x+2r\alpha,a+2b) F_{c_3}(x+3r\alpha, a+3b)\ d\mu_G(x)\\
&= 
\int_G \int_G F_{c_1}(x+y,a+b) F_{c_2}(x+2y, a+2b) F_{c_3}(x+3y, a+3b)\ d\mu_G(x) d\mu_G(y)
\end{align*} 
and so the expression in \eqref{eq50} can be simplified to
\begin{align*}
 \sum_{\substack{c_1,c_2,c_3 \in \{1,2,3\}\\ c_1+c_2+c_3 = c\ \Mod{3}}} \E_{b \in [MW]: (b,W)=1} \E_{a \in \Z/M\Z} A_{c_1,c_2,c_3}(a+b,a+2b,a+3b)
\end{align*} 
where
$$
A_{c_1,c_2,c_3}(a_1,a_2,a_3) \coloneqq \int_G \int_G F_{c_1}(x+y,a_1) F_{c_2}(x+2y,a_2) F_{c_3}(x+3y, a_3)\ d\mu_G(x) d\mu_G(y).$$
The condition $(b,W)=1$ clearly implies $(r,M,W)=1$ for any $r \in \Z/M\Z$ with $b = r\ (M)$. Conversely, if $(r,M,W)=1$, then from the Chinese remainder theorem we see that there are precisely $\frac{(M,W) \phi(W)}{\phi((M,W))}$ values of $b \in [MW]$ with $(b,W)=1$ and $b = r\ (M)$.  Thus the above expression can also be written as
$$
 \sum_{\substack{c_1,c_2,c_3 \in \{1,2,3\}\\ c_1+c_2+c_3 = c\ \Mod{3}}} \E_{a,r \in \Z/M\Z: (r,M,W)=1} A_{c_1,c_2,c_3}(a+r,a+2r,a+3r).
$$
Thus, to prove Theorem \ref{theo-erg2}(ii), we can assume for sake of contradiction that
\begin{equation}\label{assume-i}
\E_{a,r \in \Z/M\Z: (r,M,W)=1} A_{c_1,c_2,c_3}(a+r,a+2r,a+3r) \leq \eps
\end{equation}
for all $c_1,c_2,c_3\in \Z/3\Z$ satisfying $c_1+c_2+c_3=c$, and some sufficiently small $\eps>0$ depending on $\delta_1,\delta_2,\delta_3$.  Similarly, to prove Theorem \ref{theo-erg2}(i) or Theorem \ref{theo-erg2}(iii), we may assume for sake of contradiction that
\begin{equation}\label{assume-ii}
\E_{a,r \in \Z/M\Z: (r,M,W)=1} A_{1,2,3}(a+r,a+2r,a+3r) \leq \eps.
\end{equation}

We can now easily dispose of the case (i) by using the following inequality of ``Pollard-type'' \cite{pollard}.

\begin{lemma}[Pollard-type inequality]\label{le_pollard} Let $G$ be a torus of any dimension equipped with its Haar measure $\mu_G$, and let $F_1,F_2,F_3: G \to [0,1]$ be measurable functions.  Set $\delta_i \coloneqq \int_G F_i(x)\ d\mu(x)$ for $i=1,2,3$, and write $\delta \coloneqq \min(\delta_1,\delta_2,\delta_3)$.  Then, for any distinct integers $m_1,m_2,m_3$, one has 
$$ \int_G \int_G F_1(x+m_1 y) F_2(x+m_2 y) F_3(x+m_3y)\ d\mu_G(x) d\mu_G(y) \geq \frac{1}{4} \max(\delta_1+\delta_2+\delta_3-1,0)^2 $$
if $\delta_1+\delta_2+\delta_3 \leq 1 + 2\delta$, and
$$ \int_G \int_G F_1(x+m_1 y) F_2(x+m_2 y) F_3(x+m_3y) \geq \delta( \delta_1+\delta_2+\delta_3-1-\delta)$$
if $\delta_1+\delta_2+\delta_3 > 1 + 2\delta$.
\end{lemma}

\begin{proof}  We can replace the functions $F_i$ with indicator functions by the following lifting trick: if we define the subsets $A_i$ of the torus $\tilde G \coloneqq G \times (\R/\Z)^3$ for $i=1,2,3$ by the formula
$$ A_i \coloneqq \{ (x, t_1,t_2,t_3) \in \tilde G: t_i \in [0, F_i(x)]\}$$
then we see that $\delta_i = \mu_{\tilde G}(A_i)$ and
\begin{equation}\label{conv}
\begin{split}
& \int_G \int_G F_1(x+m_1 y) F_2(x+m_2 y) F_3(x+m_3 y)\ d\mu_G(x) d\mu_G(y) \\
&\quad = \int_{\tilde G} \int_{\tilde G} 1_{A_1}(\tilde x + m_1 \tilde y) 1_{A_2}(\tilde x+ m_2 \tilde y) 1_{A_3}(\tilde x+m_3 \tilde y)\ d\mu_{\tilde G}(\tilde x) d\mu_{\tilde G}(\tilde y).
\end{split}
\end{equation}
Observe that as $(\tilde x, \tilde y)$ ranges in $\tilde G \times \tilde G$, the triple $(\tilde x + m_1 \tilde y, \tilde x + m_2 \tilde y, \tilde x + m_3 \tilde y)$ ranges surjectively in the torus
$$ \{ (z_1,z_2,z_3) \in \tilde G^3: (m_3-m_2) z_1 + (m_1-m_3) z_2 + (m_2-m_1) z_3 = 0 \},$$
and furthermore that the Haar probability measure on $\tilde G \times \tilde G$ pushes forward to Haar probability measure on this torus.  Thus we can write the expression \eqref{conv} as a convolution
$$ 1_{(m_3-m_2)^{-1} A_1} * 1_{(m_1-m_3)^{-1} A_2} * 1_{(m_2-m_1)^{-1} A_3}(0)$$
where $(m_3-m_2)^{-1} A_1 := \{ \tilde x \in \tilde G: (m_3-m_2)\tilde x \in A_1\}$ has the same measure $\delta_1$ as $A_1$ (because the pushforward of Haar probability measure on $\tilde G$ by $\tilde x \mapsto (m_3-m_2) \tilde x$ is also Haar probability measure), and similarly for
$(m_1-m_3)^{-1} A_2$ and $(m_2-m_1)^{-1} A_3$.  By inner regularity we may assume that $A_1,A_2,A_3$ are all compact.  The claim now follows from \cite[Corollary 3]{tao-kemperman} (see also \cite[Theorem 1.1]{tao-kneser} for a closely related inequality).
\end{proof}

For any choice of $a,r$, we see from \eqref{fmean} and Lemma \ref{le_pollard} and the hypothesis $\delta_1+\delta_2+\delta_3>1$ of (i) that
$$ A_{1,2,3}(a+r,a+2r,a+3r) \gg 1.$$
Averaging over $a,r$ we contradict \eqref{assume-ii} if $\eps$ is small enough.

It remains to handle the critical cases (ii), (iii).  For this we use the following inverse theorem for Lemma \ref{le_pollard} that is deduced from the recent results in \cite{tao-kneser}.

\begin{theorem}[Inverse theorem]\label{invt}  Let $\delta_1,\delta_2,\delta_3>0$ be real numbers with $\delta_1+\delta_2+\delta_3=1$. Let $\kappa > 0$, and suppose that $\eps>0$ is sufficiently small depending on $\kappa$.  Let $G$ be a torus with Haar probability measure $d\mu_{G}$, and let $g_1,g_2,g_3: G \to [0,1]$ be such that
\begin{equation}\label{chicken}
\delta_i - \eps^{1/2} \leq \int_{G} g_i(x_0)\ d\mu_{G}(x_0) \leq \delta_i + \eps^{1/2}
\end{equation}
for $i=1,2,3$, and such that
\begin{equation}\label{egg}
 \int_{G} \int_{G} g_1(x_0+y_0) g_2(x_0+2y_0) g_3(x_0+3y_0)\ d\mu_{G}(x_0) d\mu_{G}(y_0) \leq \eps^{1/2}.
\end{equation}
Then there exists a non-zero element $\phi$ of the Pontragyin dual group $\hat G$ (thus $\phi: G \to\R/\Z$ is a continuous homomorphism that is not identically zero) and arcs $I_1,I_2,I_3$ in $\R/\Z$ of lengths exactly $\delta_1,\delta_2,\delta_3$, such that
\begin{align*}
g_1 &\approx_{\kappa} 1_{\phi^{-1}(I_1)} \\
g_2 &\approx_{\kappa} 1_{(2\phi)^{-1}(I_2)} \\
g_3 &\approx_{\kappa} 1_{\phi^{-1}(I_3)},
\end{align*}
where $2\phi: G \to \R/\Z$ is the map $(2\phi)(x_0) \coloneqq  2(\phi(x_0))$, and $g \approx_{\kappa} h$ denotes the estimate $\| g - h \|_{L^1(G,d\mu_{G})} \ll \kappa$.
\end{theorem}

\begin{proof}  Introduce the sets $E_1, E_3 \subset G$ by the formulae
$$ E_i \coloneqq \{ x_0 \in G: g_i(x_0) \geq \eps^{1/8} \}$$
for $i=1,3$.  On the one hand, we have the pointwise bound
$$1_{E_i} \geq g_i - \eps^{1/8}$$
and hence from \eqref{chicken}
\begin{equation}\label{eig}
 \mu_{G}(E_i) \geq \delta_i - O(\eps^{1/8})
\end{equation}
for $i=1,3$.  On the other hand, from the pointwise bound
$$ 1_{E_i} \leq \eps^{-1/8} g_i, $$
and  \eqref{egg} we have
$$ \int_{G} \int_{G} 1_{E_1}(x_0+y_0) g_2(x_0+2y_0) 1_{E_3}(x_0+3y_0)\ d\mu_{G}(x_0) d\mu_{G}(y_0) \ll \eps^{1/4}$$
or equivalently (writing $x_0+2y_0=z_0$)
$$ \int_{G} g_2(z_0) 1_{E_1} * 1_{E_3}(2z_0)\ d\mu_{G}(z_0) \ll \eps^{1/4}.$$
In particular, if we let $F$ denote the set of points $x_0 \in G$ such that $1_{E_1} * 1_{E_3}(x_0) \geq \eps^{1/8}$, then
\begin{equation}\label{g2f}
 \int_{G} g_2(z_0) 1_{F}(2z_0)\ d\mu_{G}(z_0) \ll \eps^{1/8}.
\end{equation}
Applying \cite[Corollary 1.2]{tao-kneser} and \eqref{eig}, we have
\begin{equation}\label{mugof}
 \mu_{G}(F) \geq \mu_{G}(E_1) + \mu_{G}(E_3) - O(\eps^{1/16}) \geq \delta_1+\delta_3 - O( \eps^{1/16})
\end{equation}
If one sets $F' \coloneqq  \{ z_0: 2z_0 \in F \}$, then (as $G$ is a torus) $F'$ has the same measure as $F$, thus
\begin{equation}\label{mugof-2}
 \mu_{G}(F') \geq \mu_{G}(E_1) + \mu_{G}(E_3) - O(\eps^{1/16}) \geq \delta_1+\delta_3 - O( \eps^{1/16}).
\end{equation}
In particular, since $\delta_1+\delta_2+\delta_3=1$, we obtain
$$ \int_{G \backslash F'} g_2(z_0)\ d\mu_{G}(z_0) \leq 1 - \mu_G(F') \leq \delta_2 +  O( \eps^{1/16}).$$
On the other hand, from \eqref{g2f} one has
\begin{equation}\label{f2}
 \int_{F'} g_2(z_0)\ d\mu_{G}(z_0) \ll \eps^{1/8}.
\end{equation}
From \eqref{chicken} and \eqref{f2} we get
\begin{align}\label{eqq5}
\int_{G \backslash F'} g_2(z_0)\ d\mu_{G}(z_0) = \delta_2 +  O( \eps^{1/16}).
\end{align} 
From 
\begin{align*}
\delta_2-O(\varepsilon^{1/2})\leq \int_G g_2(z_0)\ d\mu_G(z_0)\leq \int_{F'} g_2(z_0)\ d\mu_G(z_0)+\mu_G(G\setminus F')   
\end{align*}
and \eqref{mugof-2}, \eqref{f2} we get
\begin{equation}\label{mugo}
 \mu_{G}(F') = \delta_1+\delta_3 - O( \eps^{1/16}).
\end{equation}
Lastly, from \eqref{eig} and \eqref{mugof-2} we get 
\begin{align}\label{eqq4}
\mu_G(E_1)=\delta_1+O(\varepsilon^{1/16}),\quad \mu_G(E_3)=\delta_3+O(\varepsilon^{1/16}).    
\end{align}
By \eqref{mugo}, we have $\mu_G(G \setminus F')=\delta_2+O(\varepsilon^{1/16})$, which together with \eqref{f2} and \eqref{eqq5} implies that
$$ \| g_2 - 1_{F'} \|_{L^1(G)} \ll \eps^{1/16}.$$
From \eqref{eqq4} we have for $i=1,3$ that
$$ \int_{E_i} g_i(x_0)\ d\mu_{G}(x_0) \leq \delta_i + O(\eps^{1/16});$$
but by definition of $E_i$ we have
$$ \int_{G \backslash E_i} g_i(x_0)\ d\mu_{G}(x_0) \ll \eps^{1/8}.$$
Now, by \eqref{chicken} actually
\begin{align*}
\int_{E_i} g_i(x_0)\ d\mu_{G}(x_0) = \delta_i + O(\eps^{1/16}).    
\end{align*}
Comparing the two previous formulas with \eqref{eqq4}, we conclude that
\begin{equation}\label{gio}
 \| g_i - 1_{E_i} \|_{L^1(G)} \ll \eps^{1/16}.
\end{equation}

As $F$ has the same measure as $F'$, we see from \eqref{mugo}, \eqref{eqq4} that
$$ \mu_{G}(F) = \mu_{G}(E_1) + \mu_{G}(E_2) - O(\eps^{1/16}).$$
Applying\footnote{See also \cite{iliopoulou}, \cite{griesmer} for closely related results.} \cite[Theorem 1.5]{tao-kneser}, there exists a non-trivial element $\phi \in \hat G$ and arcs $I_1,I_3 \subset \R/\Z$ such that
\begin{align}\label{eq6} \mu_{G}( E_1 \triangle \phi^{-1}(I_1) ), \mu_{G}( E_3 \triangle \phi^{-1}(I_3) ) \leq \kappa^2,
\end{align}
where $\triangle$ denotes symmetric difference.  (Note from the connectedness of $G$ that $\phi(G)$ must be all of $\R/\Z$, and hence $\phi$ pushes forward $\mu_{G}$ to Haar probability measure on $\R/\Z$.  The same claim then holds for $2\phi$.) Moreover, from \eqref{eqq4} we see that necessarily $\mu_G(I_i)=\delta_i+O(\varepsilon^{1/16})$, and since $\varepsilon$ is small enough in terms of $\kappa$, we may in fact add or remove a segment from $I_i$ so that its length becomes exactly $\delta_i$ while keeping \eqref{eq6} true (with possibly $2\kappa^2$ in place of $\kappa^2$).

Combining \eqref{eq6} with \eqref{gio} we see that
$$ \| g_i - 1_{\phi^{-1}(I_i)} \|_{L^1(G)} \ll \kappa^2$$
for $i=1,3$.  From \eqref{egg} we conclude that
$$ \int_{G} \int_{G} 1_{\phi^{-1}(I_1)}(x_0) g_2(x_0+y_0) 1_{\phi^{-1}(I_3)}(x_0+2y_0)\ d\mu_{G}(x_0) d\mu_{G}(y_0) \ll \kappa^2$$
or equivalently
$$ \int_{G} g_2(z_0) 1_{\phi^{-1}(I_1)} * 1_{\phi^{-1}(I_3)}(2z_0)\ d\mu_{G}(z_0) \ll \kappa^2.$$
Let $J$ be the interval $I_1+I_3$, shrunk on both sides by $\kappa$. Then $J$ is an arc of length $\delta_1+\delta_3 - 2\kappa$ and
$$ 1_{\phi^{-1}(I_1)} * 1_{\phi^{-1}(I_3)}(x_0) \gg \kappa$$
for $x_0 \in \phi^{-1}(J)$.  We conclude that
$$ \int_{G} g_2(z_0) 1_{\phi^{-1}(J)}(2z_0)\ d\mu_{G}(z_0) \ll \kappa.$$
If we let $I_2$ denote the complement of $I_1 + I_3$, then $I_2$ is an arc of length $\delta_2$ that differs from the complement of $J$ by two arcs of total length $\kappa$, and thus
$$ \int_{G \backslash (2\phi)^{-1}(I_2)} g_2(z_0)\ d\mu_{G}(z_0) \ll \kappa.$$
Also
$$ \int_{(2\phi)^{-1}(I_2)} g_2(z_0)\ d\mu_{G}(z_0) \leq \mu_{G}( \phi^{-1}(I_2) ) = \delta_2.$$
Combining this with \eqref{chicken} we see that
\begin{align*}
\int_{(2\phi)^{-1}(I_2)}g_2(x_0)\ d\mu_G(x_0)=\delta_2+O(\kappa),
\end{align*}
so
$$ \| g_2 - 1_{(2\phi)^{-1}(I_2)} \|_{L^1(G)} \ll \kappa,$$
and the claim follows.
\end{proof}

Let $\kappa>0$ be a small absolute constant to be chosen later, and suppose $\eps>0$ is sufficiently small depending on $\kappa$. 
Suppose first that \eqref{assume-i} holds for some $c \in \Z/3\Z$.  By Markov's inequality, this implies that
for $1-O(\eps^{1/2})$ of the pairs of $(a,r) \in \Z/M\Z \times \Z/M\Z$ with $(r,M,W)=1$, and any $c_1,c_2,c_3 \in \{1,2,3\}$ with $c_1+c_2+c_3=c\ \Mod{3}$, one has
$$ A_{c_1,c_2,c_3}( a+r,a+2r,a+3r) \ll \eps^{1/2},$$
Applying Theorem \ref{invt}, we conclude that for such pairs $(a,r)$, there exists a non-trivial element $\phi_{a,r;c_1,c_2,c_3} \in \hat G$ and arcs $I_{a,r;c_1,c_2,c_3,i} \subset \R/\Z$ for $i=1,2,3$ and any $c_1,c_2,c_3 \in \{1,2,3\}$ with $c_1+c_2+c_3=c\ \Mod{3}$, one has
\begin{align*}
F_{c_1}(\cdot,a+r) &\approx_{\kappa} 1_{\phi^{-1}_{a,r;c_1,c_2,c_3}(I_{a,r;c_1,c_2,c_3,1})} \\
F_{c_2}(\cdot,a+2r) &\approx_{\kappa} 1_{(2\phi_{a,r;c_1,c_2,c_3})^{-1}(I_{a,r;c_1,c_2,c_3,2})} \\
F_{c_3}(\cdot,a+3r) &\approx_{\kappa} 1_{\phi^{-1}_{a,r;c_1,c_2,c_3}(I_{a,r;c_1,c_2,c_3,3})}.
\end{align*}
From \eqref{fmean} we see that the arc $I_{a,r;c_1,c_2,c_3,i}$ has length $\delta_i + O(\kappa)$ for $i=1,2,3$.

Now we start removing the dependence of $\phi_{a,r;c_1,c_2,c_3}$ on the various parameters $a,r,c_1,c_2,c_3$.  The key lemma is the following.

\begin{lemma}\label{chak}  Let $0 < \sigma < 1/2$, and suppose that $\delta>0$ is sufficiently small depending on $\sigma$.  Let $\phi_1, \phi_2 \in \hat G$ be non-trivial, and let $I_1, I_2 \subset \R/\Z$ be arcs of length between $\sigma$ and $1-\sigma$.  Suppose that $1_{\phi^{-1}_1(I_1)} \approx_{\delta} 1_{\phi^{-1}_2(I_2)}$.  Then we have $\phi_2 = \pm \phi_1$.
\end{lemma}

\begin{proof}  By hypothesis, we have
\begin{align}\label{eqq7}
\int_{G} 1_{\phi^{-1}_1(I_1)} 1_{\phi^{-1}_2(I_2)}\ d\mu_{G} = \mu_{G}( \phi^{-1}_1(I_1) ) + O(\delta) = m(I_1) + O(\delta)
\end{align}
and similarly for $\phi_2$ and $I_2$, where $m$ denotes Lebesgue measure on $\R/\Z$.  In particular $m(I_2) = m(I_1) + O(\delta)$.
By Fourier inversion, the left-hand side of \eqref{eqq7} is equal to
$$ \sum_{\substack{n,m \in \Z\\ n \phi_1 + m \phi_2 = 0}} \check 1_{I_1}(n) \check 1_{I_2}(m)$$
where
$$ \check 1_{I_1}(n) \coloneqq \int_{\R/\Z} 1_{I_1}(\alpha) e(-n\alpha)\ d\alpha$$
and similarly for $\check 1_{I_2}(m)$.  On the other hand, as $G$ is connected, the Pontryagin dual $\hat G$ is torsion-free, so for each $n$ there is at most one $m$ such that $n \phi_1 + m \phi_2 = 0$ and vice versa.  If $\phi_1$ is not an integer multiple of $\phi_2$, then we may omit the $n=1$ terms, and conclude from Cauchy--Schwarz that
$$ \left( \sum_{n \in \Z \backslash \{1\}} |\check 1_{I_1}(n)|^2\right)^{1/2}
\left( \sum_{m \in \Z} |\check 1_{I_2}(m)|^2\right)^{1/2} \geq m(I_1) + O(\delta).$$
On the other hand, from the Plancherel identity one has
$$ \sum_{m \in \Z} |\check 1_{I_2}(m)|^2 = m(I_2)$$
and (by explicit computation of $\check 1_{I_1}(1)$)
$$ \sum_{n \in \Z \backslash \{1\}} |\check 1_{I_1}(n)|^2 = m(I_1) - |\check 1_{I_1}(1)|^2 \leq m(I_1) - c_\sigma$$
for some quantity $c_\sigma>0$ depending only on $\sigma$.  For $\delta$ small enough, this leads to a contradiction.  Thus $\phi_1$ is an integer multiple of $\phi_2$, and similarly $\phi_2$ is an integer multiple of $\phi_1$; thus $\phi_2 = \pm \phi_1$ as claimed.
\end{proof}

From this lemma, we see that for each $c_1 \in \Z/3\Z$ and $a \in \Z/M\Z$, there is at most one non-trivial $\phi_{a;c_1} \in \hat G$ up to sign such that $F_{c_1}(\cdot,a) \approx_{\kappa} 1_{\phi^{-1}_{a,c}(I_{c_1})}$ for some arc $I_{c_1}$ of length $\delta_{c_1}$.  Select such a $\phi_{a;c_1}$ for each $a,c_1$ (or select $\phi$ arbitrarily if no such $I_{c_1}$ exists).  Then for $1-O(\eps^{1/2})$ of the pairs of $(a,r)$ with $(r,M,W)=1$, and any $c_1,c_2,c_3$ with $c_1+c_2+c_3=c$, we have
\begin{align*}
\phi_{a,r;c_1,c_2,c_3} &= \pm \phi_{a+r;c_1} \\
2\phi_{a,r;c_1,c_2,c_3} &= \pm \phi_{a+2r; c_2} \\
\phi_{a,r;c_1,c_2,c_3} &= \pm \phi_{a+3r;c_3}.
\end{align*}
In particular, for such a pair $(a,r)$ we have
$$ \phi_{a+r;c_1} = \pm \phi_{a+3r;c_3} $$
for any $c_1,c_3 \in \{1,2,3\}$ (choosing $c_2$ to be congruent to $c - c_1 - c_3$ modulo $3$), which implies in particular that $\phi_{a+r;c_1}$ does not depend on $c_1$ up to sign.  Thus, we can actually find a non-trivial $\phi_a \in \hat G$ for all $a \in \Z/M\Z$, such that one has
$$ \phi_{a+r} = \pm \phi_{a+3r}; \quad 2 \phi_{a+r} = \pm \phi_{a+2r}$$
for $1-O(\eps^{1/2})$ of the pairs of $(a,r)$ with $(r,M,W)=1$.  Replacing $(a,r)$ by $(a-r,r)$ and $(a+2r,-r)$, we also see that for $1-O(\eps^{1/2})$ of such pairs, we simultaneously have
$$ \phi_{a} = \pm \phi_{a+2r}; \quad 2 \phi_{a} = \pm \phi_{a+r}$$
and
$$ \phi_{a+r} = \pm \phi_{a-r}; \quad 2 \phi_{a+r} = \pm \phi_{a}$$
which implies in particular that
$$ 4 \phi_a = \pm \phi_a.$$
But this is impossible since $\hat G$ is torsion-free and $\phi_a$ is non-trivial.  This proves Theorem \ref{theo-erg2}(ii).

Now we turn to Theorem \ref{theo-erg2}(iii).  With $\kappa$ and $\eps$ as above, we now assume instead that $\delta_1 \neq \delta_3$ and that \eqref{assume-ii} holds.  Again using Markov's inequality followed by Theorem \ref{invt}, we now conclude that for $1-O(\eps^{1/2})$ of the pairs of $(a,r) \in \Z/M\Z \times \Z/M\Z$ with $(r,M,W)=1$, one has a non-trivial element $\phi_{a,r} \in \hat G$ and arcs $I_{a,r;i} \subset \R/\Z$ for $i=1,2,3$ such that
\begin{align*}
F_{1}(\cdot,a+r) &\approx_{\kappa} 1_{\phi^{-1}_{a,r}(I_{a,r;1})} \\
F_{2}(\cdot,a+2r) &\approx_{\kappa} 1_{(2\phi_{a,r})^{-1}(I_{a,r;2})} \\
F_{3}(\cdot,a+3r) &\approx_{\kappa} 1_{\phi^{-1}_{a,r}(I_{a,r;3})}.
\end{align*}
From \eqref{fmean} we see that the arc $I_{a,r;i}$ has length $\delta_i + O(\eps^{1/2})$ for $i=1,2,3$.

Applying Lemma \ref{chak}, we conclude that one can find non-trivial characters $\phi_a^{(i)} \in \hat G$ for $a \in \Z/M\Z$, $i=1,2,3$ such that for $1-O(\eps^{1/2})$ of the pairs of $(a,r) \in \Z/M\Z \times \Z/M\Z$ with $(r,M,W)=1$, we have
\begin{align}
\phi_{a,r} &= \pm \phi_{a+r}^{(1)}\label{ar1} \\
2\phi_{a,r} &= \pm \phi_{a+2r}^{(2)}\label{ar2} \\
\phi_{a,r} &= \pm \phi_{a+3r}^{(3)}\label{ar3}
\end{align}
so in particular
$$ \phi_{a+2r}^{(2)} = \pm 2 \phi_{a+r}^{(1)}.$$
Replacing $a$ by $a-r$, we conclude that for $1-O(\eps^{1/2})$ of the above pairs $(a,r)$, we have
$$ \phi_{a+r}^{(2)} = \pm 2 \phi_{a}^{(1)}.$$
This implies that for $1-O(\eps^{1/2})$ of the triples $(a,r,r') \in \Z/M\Z \times \Z/M\Z \times \Z/M\Z$ with $(r,M,W), (r',M,W) = 1$, we have
$$ \phi_{a+r}^{(2)} = \pm 2 \phi_{a}^{(1)}$$
and
$$ \phi_{a+r}^{(2)} = \pm 2 \phi_{a+r-r'}^{(1)}$$
which implies (by the torsion-free nature of $\hat G$) that
$$ \phi_{a+r-r'}^{(1)} = \pm \phi_{a}^{(1)}.$$
Iterating this two more times, we see that for $1-O(\eps^{1/2})$ of the septuples $(a,(r_i)_{i=1}^{6}) \in (\Z/M\Z)^7$ with $(r_i,M,W)=1$ for $1\leq i\leq 6$, one has
$$ \phi_{a+r_1-r_2+r_3-r_4+r_5-r_6}^{(1)} = \pm \phi_{a}^{(1)}.$$
For any $h \in \Z/M\Z$, the number of sextuples $(r_i)_{i=1}^{6} \in (\Z/M\Z)^6$ with $r_1-r_2+r_3-r_4+r_5-r_6=2h$ and $(r_i,M,W)=1$ for $1\leq i\leq 6$ can be computed using the Chinese remainder theorem to be comparable (up to absolute constants) to the quantity $\frac{1}{M} \left(\frac{\phi((M,W))}{(M,W)} M \right)^6$.  (The factor of $2$ here is needed to avoid the parity obstruction that $r_1-r_2+r_3-r_4$ is necessarily even if $(M,W)$ is even.) On the other hand, the number of representations of $r_1-r_2+r_3-r_4+r_5-r_6=2h$ where $r_i\in (\mathbb{Z}/M\mathbb{Z})$ and $(r_i,M,W)=1$, and $r_1$ (say) belongs to an exceptional set of size $O(\varepsilon^{1/2}|\{r\in \mathbb{Z}/M\mathbb{Z}:\,\, (r,M,W)=1\}|)$ is bounded by\footnote{The validity of this bound follows from the fact that the number of representations $-r_2+r_3-r_4+r_5-r_6=h'$ with $r_i\in \mathbb{Z}/M\mathbb{Z}$ and $(r_i,M,W)=1$ is uniformly $\ll \frac{1}{M}\left(\frac{\phi((M,W))}{(M,W)} M \right)^5$.} $\ll \varepsilon^{1/2}\frac{1}{M}\left(\frac{\phi((M,W))}{(M,W)} M \right)^6$.

From this and a double counting argument, we see that for $1-O(\eps^{1/2})$ of the pairs $(a,h) \in (\Z/M\Z)^2$, we have
$$ \phi_{a+2h}^{(1)} = \pm \phi_{a}^{(1)}.$$
We conclude that there exists a non-zero element $\phi$ of $\hat G$ such that
$$ \phi_{2a}^{(1)} = \pm \phi$$
for $1-O(\eps^{1/2})$ of $a \in \Z/M\Z$.  Inserting this back into \eqref{ar1}, \eqref{ar2}, \eqref{ar3} and double counting, we conclude that
$$ \phi_{2a}^{(2)} = \pm 2 \phi$$
and
$$ \phi_{2a}^{(3)} = \pm \phi$$
for $1-O(\eps^{1/2})$ of $a \in \Z/M\Z$.  This implies that for $1-O(\eps^{1/2})$ of $a \in \Z/M\Z$, we can find arcs $I_{a;1}, I_{a;2}, I_{a;3}$ in $\R/\Z$ such that
\begin{align*}
F_{1}(\cdot,a) &\approx_{\kappa} 1_{\phi^{-1}(I_{a;1})} \\
F_{2}(\cdot,a) &\approx_{\kappa} 1_{(2\phi)^{-1}(I_{a;2})} \\
F_{3}(\cdot,a) &\approx_{\kappa} 1_{\phi^{-1}(I_{a;3})}.
\end{align*}
Fix such an $a$.  From \eqref{fmean} we see that each arc $I_{a;i}$ has length $\delta_i + O(\kappa)$ for $i=1,2,3$. As $F_1+F_2+F_3=1$, we have
$$ 1 \approx_{\kappa} 1_{\phi^{-1}(I_{a;1})} + 1_{(2\phi)^{-1}(I_{a;2})} + 1_{\phi^{-1}(I_{a;3})}.$$
Since $\phi$ pushes forward $\mu_G$ to Haar measure $m$ on $\R/\Z$, we conclude that
$$ \int_{\R/\Z} |1_{I_{a;1}}(\theta) + 1_{I_{a;2}}(2\theta) + 1_{I_{a;3}}(\theta) - 1|\ dm(\theta) \ll \delta,$$
which implies that the set $I_{a;1} \cup I_{a;3}$ differs by at most $O(\delta)$ in measure from the set $\{ \theta \in \R/\Z: 2 \theta \not \in I_{a;2}\}$.  But since $I_{a;2}$ is an arc length $\delta_2 + O(\kappa)$, the set $\{ \theta \in \R/\Z: 2 \theta \not \in I_{a;2}\}$ is the union of two arcs of length $\frac{1-\delta_2}{2} + O(\kappa)$, separated from each other by distance $\frac{\delta_2}{2} + O(\kappa)$.  Since $\delta_1 \neq \delta_3$ and $\delta_1+\delta_2+\delta_3=1$, $\delta_1$ and $\delta_3$ are both distinct from $\frac{1-\delta_2}{2}$.  As $I_{a;1}$ and $I_{a;3}$ are arcs of length $\delta_1+O(\kappa)$ and $\delta_3+O(\kappa)$, this leads to a contradiction for $\kappa$ small enough.  This proves Theorem \ref{theo-erg2}(iii).

\section[The main theorem for k>3]{The main theorem for $k\geq 4$.}\label{sec: higher}

We now prove Theorem \ref{theo-erg2}(iv).  By reducing the functions $F_i$ by an appropriate scalar multiple, we may assume that $\delta_1 = \dots = \delta_k = \delta$ for some $\delta > c_k$.  From \eqref{erg2} and the triangle inequality, we have
$$ \lim_{P \to \infty} \E_{d \leq P: (d,W)=1} \int_X F_i( T^{id} x)\ d\mu(x) = \delta $$
for any $1 \leq i \leq k$, and also
$$ \lim_{P \to \infty} \E_{d \leq P: (d,W)=1} \int_X F_i( T^{id} x) F_{i'}( T^{i'd} x)\ d\mu(x) = \delta^2 $$
for any $1 \leq i < i' \leq k$ (this can be seen by first changing variables from $x$ to $T^{id} x$, pulling the $d$ sum inside the integral, and using \eqref{erg2} and the triangle inequality).  This implies that
\begin{equation}\label{md1}
 \lim_{P \to \infty} \E_{d \leq P: (d,W)=1} \int_X (1-F_i( T^{id} x))\ d\mu(x) = 1-\delta 
\end{equation}
and
\begin{equation}\label{md2}
 \lim_{P \to \infty} \E_{d \leq P: (d,W)=1} \int_X (1-F_i( T^{id} x)) (1-F_{i'}( T^{i'd} x))\ d\mu(x) = (1-\delta)^2.
\end{equation}
Now we bound triple correlations.

\begin{lemma}  For $1 \leq i < i' < i'' \leq k$, one has
$$ \lim_{P \to \infty} \E_{d \leq P: (d,W)=1} \int_X (1-F_i( T^{id} x)) (1-F_{i'}( T^{i'd} x)) (1-F_{i''}( T^{i''d} x))\ d\mu(x) \leq \frac{3}{4} (1-\delta)^2.$$
\end{lemma}

\begin{proof}  By inclusion-exclusion, it suffices to show that
$$ \lim_{P \to \infty} \E_{d \leq P: (d,W)=1} \int_X (1-F_i( T^{id} x)) (1-F_{i'}( T^{i'd} x)) F_{i''}( T^{i''d} x)\ d\mu(x) \geq \frac{1}{4} (1-\delta)^2.$$
By repeating the arguments\footnote{Here it is essential that there are only three factors in the average considered here, so that the average is of ``complexity one'' and can thus be controlled by the Kronecker factor.  The same is not true for the original average \eqref{dw}, but we will not need to directly pass to characteristic factors for that average.} of the previous section, to prove this it suffices to do so when $X = G \times \Z/M\Z$ with $G$ a torus with shift $T(x,a) = (x+\alpha,a+1)$.  It then suffices to establish the lower bound
$$ \int_G \int_G (1-F_{i,a+ir}(x+iy)) (1-F_{i',a+i'r}(x+i'y)) F_{i'',a+i''r}(x+i''y)\ d\mu_G(x) d\mu_G(y) \geq \frac{1}{4} (1-\delta)^2$$
for all $a,r \in \Z/M\Z$, where the $F_{i,a}: G \to [0,1]$ are measurable functions of mean $\delta$.  But this follows from Lemma \ref{le_pollard} (noting that $\delta > c_k > 1/2$ and hence $(1-\delta)+(1-\delta)+\delta-1 \leq 1 + 2 \min(\delta,1-\delta)$).
\end{proof}

Let $\tilde X$ be the space $X \times [0,1]^k$ with the product measure $d\mu dt_1 \dots dt_k$, and for each $1 \leq i \leq k$, let $E_i \subset \tilde X$ denote the set
$$ E_i \coloneqq \{ (x,t_1,\dots,t_k) \in \tilde X: t_i > F_i(x) \},$$
then from the above lemma we have
$$ \lim_{P \to \infty} \E_{d \leq P: (d,W)=1} \int_{\tilde X} 1_{E_i}( T^{id} x, t ) 1_{E_{i'}}( T^{i'd} x, t ) 1_{E_{i''}}( T^{i''d} x, t )\ d\mu(x) dt \leq \frac{3}{4} (1-\delta)^2.$$
Hence, if $N(d,x,t)$ denotes the counting function
$$ N(d,x,t) \coloneqq \sum_{i=1}^k 1_{E_i}(T^{id} x, t)$$
then on summing the preceding assertion in $i,i',i''$ we obtain
$$ \lim_{P \to \infty} \E_{d \leq P: (d,W)=1} \int_{\tilde X} \binom{N(d,x,t)}{3}\ d\mu(x)\ dt \leq \binom{k}{3} \frac{3}{4} (1-\delta)^2.$$
Applying similar arguments to \eqref{md1}, \eqref{md2} we obtain
$$ \lim_{P \to \infty} \E_{d \leq P: (d,W)=1} \int_{\tilde X} \binom{N(d,x,t)}{2}\ d\mu(x)\ dt = \binom{k}{2} (1-\delta)^2$$
and
$$ \lim_{P \to \infty} \E_{d \leq P: (d,W)=1} \int_{\tilde X} \binom{N(d,x,t)}{1}\ d\mu(x)\ dt = \binom{k}{1} (1-\delta).$$
On the other hand, if \eqref{dw} fails, then
$$ \lim_{P \to \infty} \E_{d \leq P: (d,W)=1} \int_{\tilde X} 1_{N(d,x,t) = 0}\ dt \ll \eps$$
for any given $\eps$.

Next, note that for any integer $1\leq a\leq k$ we have the inequality
\begin{align*}
(N(d,x,t)-1)(N(d,x,t)-a)(N(d,x,t)-a+1))\geq 0    
\end{align*}
whenever $N(d,x,t) \neq 0$, since $N(d,x,t)$ is then an integer from $1$ to $k$.  On using the identities\footnote{More generally one has $x^n = \sum_{k=1}^n k! S(n,k) \binom{x}{k}$, where $S(n,k)$ are the Stirling numbers of the second kind.}
\begin{align*}
x^3&=6\binom{x}{3}+6\binom{x}{2}+\binom{x}{1},\\
x^2&=2\binom{x}{2}+\binom{x}{1},
\end{align*}
this gives
\begin{align*}
6 \binom{N(d,x,t)}{3}+(6-4a)\binom{N(d,x,t)}{2}+(a^2-a)\binom{N(d,x,t)}{1}-a(a-1)\geq 0.    
\end{align*}
Averaging in $d,x,t$ and using the previous estimates, we conclude that
\begin{align}\label{eq30}
\left(\frac{9}{2}\binom{k}{3}+(6-4a)\binom{k}{2}\right)(1-\delta)^2+(a^2-a)k(1-\delta)-a(a-1)\geq -O_k(\eps).    
\end{align}
We take $a=a_k=\lceil\frac{3k+2}{4}\rceil$ here (this turns out to be the optimal choice). Then this becomes exactly the same quadratic equation as in the definition of $c_k$ in Theorem \ref{theo2}, which gives the desired contradiction if $\eps$ is small enough. This concludes the proof of Theorem \ref{theo-erg2}(iv). 

\section[Obstructions for higher values of k]{Obstructions for higher values of $k$}

In this section we give some limitations as to how much the value $c_k=1-\frac{1}{k-\frac{4}{3}+o(1)}$ appearing in Theorem \ref{theo-erg2}(iv) may be lowered for large values of $k$.

\begin{lemma}\label{2m2}  Let $m \geq 3$ be a natural number.  Then there exist shifts $a_1,\dots,a_{m^2} \in \R/\Z$ such that the ``strips''
$$ S_i \coloneqq \{ (x,y) \in (\R/\Z)^2: x + iy \in a_i + [0, \frac{2}{m}] \hbox{ mod } 1 \}$$
for $i=1,\dots,m^2$ cover the entire torus $(\R/\Z)^2$.
\end{lemma}

\begin{proof}  We set $a_1=\dots=a_m=0$.  Then for any $y \in [\frac{1}{m}, \frac{2}{m}]$, the strips $S_1,\dots,S_m$ intersect the circle $\{ (x,y): x \in \R/\Z\}$ in arcs $\{ (x,y): x \in [0, \frac{2}{m}] - iy \}$.  These $m$ arcs have length $\frac{2}{m}$, with consecutive arcs intersecting in an arc of length at most $\frac{1}{m}$.  The union of these $m$ arcs is then an arc of length at least $1$ and thus covers the whole circle.  Thus we have the inclusion
$$ (\R/\Z) \times [\frac{1}{m}, \frac{2}{m}] \subset S_1 \cup \dots \cup S_m.$$
By applying a ``Galilean transformation'', we conclude that for any $1 \leq j < m$, if we define $a_{jm+i} := \frac{j}{m} i$ for $i=1,\dots,m$, then for $y \in \frac{j}{m} + [\frac{1}{m}, \frac{2}{m}] = [\frac{j+1}{m}, \frac{j+2}{m}]$, the strips $S_{jm+1},\dots,S_{jm+m}$ intersect the circle $\{ (x,y): x \in \R/\Z\}$ in overlapping arcs of total length at least $1$, so that
$$ (\R/\Z) \times [\frac{j+1}{m}, \frac{j+2}{m}] \subset S_{jm+1} \cup \dots \cup S_{jm+m}.$$
 Taking the union over all $j=0,\dots,m-1$, we obtain the claim.
\end{proof}

\begin{corollary}\label{klow}  Let $k \geq 9$.  In Theorem \ref{theo-erg2}, one cannot replace $c_k$ with any quantity lower than $1-\frac{2}{\lfloor \sqrt{k} \rfloor}$.  (For $3\leq k<9$ this conclusion is vacuously true.)
\end{corollary}

\begin{proof}  Set $m \coloneqq \lfloor\sqrt{k}\rfloor$, so that $m \geq 3$ and $m^2 \leq k$.   Let $a_1, \dots, a_{m^2} \in \R/\Z$ be as in the preceding lemma, set $a_i$ arbitrarily for $m^2 < i \leq k$, and let $I_i$ be the complement of $a_i + [0, \frac{2}{m}] \hbox{ mod } 1$ in $\R/\Z$ for $i=1,\dots,k$.  By the above lemma, we have
$$ \prod_{i=1}^k 1_{I_i}( x + iy ) = 0 $$
for all $x,y \in \R/\Z$.  If we then set $X$ to be the unit circle $\R/\Z$ with Haar measure and an irrational shift $T: x \mapsto x +\alpha$ for some irrational $\alpha \in \R/\Z$, and set $F_i \coloneqq 1_{I_i}$, we obtain the claim (with $\delta_i = 1 - \frac{2}{m} = 1 - \frac{2}{\lfloor \sqrt{k} \rfloor}$ for $i=1,\dots,k$).
\end{proof}

Clearly, any quantitative improvement in the covering construction in Lemma \ref{2m2} would lead to a stronger lower bound on the optimal value of $c_k$ in Corollary \ref{klow}.  We do not know however whether the optimal value behaves like $1-\frac{1}{k}$, like $1-\frac{1}{\sqrt{k}}$, or has some intermediate behaviour.

\section{Proofs of the applications}\label{sec: apps}

\begin{proof}[Proof of Theorem \ref{theo_largest}] Consider the sets $Q_{\alpha, \beta}=\{n\in \mathbb{N}: n^{\alpha}<P^{+}(n)<n^{\beta}\}$ with $0\leq \alpha <\beta\leq 1$. These sets are always stable since for any prime $p$ we have $1_{Q_{\alpha, \beta}}(pn)=1_{Q_{\alpha, \beta}}(n)+O(1_{P^{+}(n)\in [n^{\alpha}, (pn)^{\alpha}]\cup [n^{\beta}, (pn)^{\beta}]})+O(1_{p>n^{\alpha}})$ and, after taking expectations over $n\leq x$, the $O(\cdot)$ term becomes negligible (as follows for instance from the continuity of the Dickman function). Also, $Q_{\alpha, \beta}$ is uniformly distributed in short intervals with density $\rho(1/\beta)-\rho(1/\alpha)$, since for $x/\log x\leq n\leq x$ we have
\begin{align*}
 1_{Q_{\alpha, \beta}}(n)=1_{P^{+}(n)\leq x^{\beta}}-1_{P^{+}(n)\leq x^{\alpha}}+O(1_{P^{+}(n)\in [(x/\log x)^{\alpha},x^{\alpha}]\cup [(x/\log x)^{\beta},x^{\beta}]}),   
\end{align*}
and the $O(\cdot)$ term is negligible, whereas $1_{P^{+}(n)\leq x^{\alpha}}$ is a real-valued multiplicative function, so by \cite[Lemma 3.4]{tera-binary} we have
\begin{align*}
\int_{0}^{x}\big|\mathbb{E}_{y\leq n\leq y+H,n=b\Mod{q}}1_{P^{+}(n)\leq x^{\alpha}}-\rho(1/\alpha)\big|\, dy=o(1),
\end{align*}
and the same holds with $\alpha$ replaced by $\beta$. Now, since $d(Q_{0,\alpha})+d(Q_{\alpha,\beta})+d(Q_{\beta,1})=1$, $d(Q_{0,\alpha}\cup Q_{\alpha, \beta}\cup Q_{\beta,1})=1$, and
$$
d(Q_{0,\alpha}) = \rho(1/\alpha) \neq 1 - \rho(1/\beta) = d(Q_{\beta,1})$$
by hypothesis, we conclude from Theorem \ref{theo1.5} that 
\begin{align*}
d_{-}((Q_{0,\alpha}-1)\cap (Q_{\alpha, \beta}-2)\cap (Q_{\beta,1}-3))>0    
\end{align*}
whenever $\rho(1/\alpha)\neq 1-\rho(1/\beta)$, and the positivity of the first density in Theorem \ref{theo_largest} follows. The positivity of the second density is proven completely symmetrically.
\end{proof}

\begin{proof}[Proof of Theorem \ref{theo_threeprimes}] We know from the proof of Theorem \ref{theo_largest} that $\{n\in \mathbb{N}: P^{+}(n)>n^{\gamma}\}$ is a stable set that is uniformly distributed in short intervals with density $1-\rho(1/\gamma)$. Thus, as long as $3(1-\rho(1/\gamma))>1$, we can apply Theorem \ref{theo1} to obtain the desired conclusion. But $3(1-\rho(1/\gamma))>1$ holds exactly when $\gamma <e^{-1/3}$, as wanted.
\end{proof}

\begin{proof}[Proof of Theorem \ref{theo_kprimes}] Employing Theorem \ref{theo2},  we only need to show that if $c_k$ are as in that theorem, then $1-\rho(1/\gamma_k)>c_k$ for $k=4,5$, and this is true by a numerical computation.
\end{proof}

\begin{proof}[Proof of Theorem \ref{theo_comparison}] By applying Theorem \ref{theo_largest} for any $\alpha,\beta$ satisfying $\rho(1/\alpha)\neq 1-\rho(1/\beta)$, we already know that
\begin{align}\label{eq17}\begin{split}
&d_{-}(n\in \mathbb{N}:P^{+}(n+1)<P^{+}(n+2)<P^{+}(n+3))>0,\\
&d_{-}(n\in \mathbb{N}:P^{+}(n+1)>P^{+}(n+2)>P^{+}(n+3))>0. 
\end{split}
\end{align}
We prove the positivity of the first density in Theorem \ref{theo_comparison}; the second one is proven completely symmetrically. We follow the strategy of \cite[Corollary 2.8]{mrt-sign}. Suppose for a contradiction that we had
\begin{align*}
\lim_{l\to \infty}\mathbb{E}_{n\leq x_l}1_{P^{+}(n+1)<P^{+}(n+2)<P^{+}(n+3)>P^{+}(n+4)}=0    
\end{align*}
for some sequence $(x_l)_{l \in \mathbb{N}}$ tending to infinity. Let 
\begin{align*}
\mathcal{S}\coloneqq \{n\in \mathbb{N}: P^{+}(n+1)<P^{+}(n+2)<P^{+}(n+3)\}.    
\end{align*}
Then as $l\to \infty$ we have
\begin{align*}
\mathbb{E}_{n\leq x_l}1_{n\in \mathcal{S},n+1\not \in \mathcal{S}}=o(1).    
\end{align*}
Iterating this, for any $H\in \mathbb{N}$, we see that for almost all $n\leq x_l$ we have
\begin{align*}
\mathbb{E}_{n\leq x_l}1_{n\in \mathcal{S}}1_{n+1\not \in \mathcal{S}\,\textrm{or}\, n+2\not \in \mathcal{S}\,\textrm{or}\ldots\textrm{or}\, n+H\not \in \mathcal{S}}=o(1).    
\end{align*}
In particular, this yields
\begin{align*}
\mathbb{E}_{n\leq x_l}1_{n\not \in \mathcal{S}}+\mathbb{E}_{n\leq x_l}1_{n+1,\ldots, n+H\in \mathcal{S}}\geq 1-o(1)    
\end{align*}
as $l \to \infty$. By \eqref{eq17}, we must then have 
\begin{align*}
\mathbb{E}_{n\leq x_l}1_{n+1,\ldots, n+H\in \mathcal{S}}\geq c - o(1)    
\end{align*}
for some $c>0$ independent of $H$ and for all large enough $l$. However, for any $\varepsilon>0$ we have 
\begin{align}\label{eq18}\begin{split}
&(\mathcal{S}-1) \cap \cdots \cap (\mathcal{S}-H)\\
&\subset \{n\in \mathbb{N}:P^{+}(n+1)\leq n^{\varepsilon}\}\cup \{n\in \mathbb{N}: P^{+}(n+h)>n^{\varepsilon}\quad \textnormal{for all}\quad 2\leq h\leq H\}.    
\end{split}
\end{align}
The density of the first set on the right-hand side of \eqref{eq18} over $n \leq x_l$ is $\rho(1/\varepsilon) + o(1)$, whereas by the Matom\"aki--Radziwi\l{}\l{} theorem the density of the second set is $\leq \varepsilon + o(1)$ as soon as $H$ is large enough in terms of $\varepsilon$. Thus
\begin{align*}
\mathbb{E}_{n\leq x_l}1_{n,n+1,\ldots, n+H\in \mathcal{S}}\leq \rho(1/\varepsilon)+\varepsilon + o(1)
\end{align*}
for all large enough $H$, and letting $\varepsilon\to 0$ we get the desired contradiction.
\end{proof}

\begin{proof}[Proof of Theorem \ref{theo_omega}] We prove the theorem for $\omega(n)$; the case of $\Omega(n)$ is similar (and in fact slightly simpler). We first note that the sets $A_{a}\coloneqq \{n\in \mathbb{N}: \omega(n)\equiv a \Mod{3}\}$ are weakly stable; indeed, for any prime $p\nmid n$ we have
\begin{align*}
1_{A_a}(n)=1_{A_{a+1}}(pn).    
\end{align*}
Also, we can represent $1_{A}(n)$ as a linear combination of $1$-bounded multiplicative functions by the Fourier expansion
\begin{align}\label{eq17a}
1_{A}(n)=\frac{1}{3}\sum_{j=0}^{2} \zeta^{-aj} \zeta^{\omega(n)j}    
\end{align}
where $\zeta \coloneqq e\left(\frac{1}{3} \right)$.
The constant function $1/3$ is certainly uniformly distributed in short intervals with density $1/3$.   The multiplicative function $n\mapsto \zeta^{\omega(n)}$ is uniformly distributed in short intervals with density $0$ thanks to \cite[Theorem A.1]{mrt-average} since
\begin{align}\label{eq18a}
\inf_{|t|\leq x} \sum_{p \leq x} \frac{1 - \textnormal{Re}( \zeta^{\omega(p)} \overline{\chi(p)p^{it}})}{p} \gg_{\chi}\log \log x \end{align}
for every Dirichlet character $\chi$ by the Vinogradov--Korobov zero-free region for Dirichlet $L$-functions. Thus $A_a$ itself is uniformly distributed in short intervals with density $1/3$.

Our objective is to show that $(A_{a_1}-1) \cap (A_{a_2}-2) \cap (A_{a_3}-3)$ has positive lower density for any $a_1,a_2,a_3\in \mathbb{Z}/3\mathbb{Z}$.
By modifying the first part of the proof of Theorem \ref{weak-stable}, it suffices to show that for every function $1\leq \omega(X)\leq X$ tending to infinity we have
\begin{equation}\label{eq19}
\mathbb{E}_{x/\omega(x)\leq n\leq x}^{\log}1_{A_{a_1}}(n+1)1_{A_{a_2}}(n+2)1_{A_{a_3}}(n+3)\gg 1.    
\end{equation}
The left-hand side of \eqref{eq19} can be expanded using \eqref{eq17a} as
\begin{equation}\label{eq20}
\frac{1}{27}\sum_{\substack{j_1,j_2,j_3\in \{0,1,2\}}} \zeta^{-(a_1j_1+a_2j_2+a_3j_3)} C_{j_1,j_2,j_3} 
\end{equation}
where
$$ C_{j_1,j_2,j_3} \coloneqq \mathbb{E}_{x/\omega(x)\leq n\leq x}^{\log} \zeta^{j_1\omega(n+1)+j_2\omega(n+2)+j_3\omega(n+3)}.    
$$
Clearly $C_{0,0,0} = 1$.  From \cite[Theorem 1.3]{tao-chowla} and \eqref{eq18a} we also have $C_{j_1,j_2,j_3}=o(1)$ when one or two of the $j_1,j_2,j_3$ vanish.  Finally, from the weak form of the logarithmic Elliott conjecture from \cite[Corollary 1.6]{tt-elliott} combined with \eqref{eq18a} we also see that $C_{j_1,j_2,j_3} = o(1)$ whenever $j_1 + j_2 + j_3 \neq 0 \Mod{3}$.  Finally we have $C_{2,2,2} = \overline{C_{1,1,1}}$.  Putting all this together, we can write the left-hand side of \eqref{eq19} as
\begin{equation}\label{eq21}
 \frac{1}{27} \left( 1 + 2 \textnormal{Re}(\zeta^{-a_1-a_2-a_3} C_{1,1,1} )\right) + o(1).
\end{equation}
For $c\in \{0,1,2\}$, let 
\begin{align*}
\delta_c\coloneqq  \mathbb{E}_{x/\omega(x)\leq n\leq x}^{\log}1_{\omega(n+1)+\omega(n+2)+\omega(n+3)\equiv c\Mod{3}}. 
\end{align*}
Then \eqref{eq21} can be rewritten as
\begin{align}\label{eq31}
\frac{1}{27} \left( 1 + 2\textnormal{Re}(\zeta^{-a} \delta_0 + \zeta^{-a+1} \delta_1 + \zeta^{-a+2} \delta_2) + o(1) \right)    
\end{align}
where $a \coloneqq a_1 + a_2 + a_3\ \Mod{3}$.  Since $\textnormal{Re}(\zeta^{-a})=1$ if $a\equiv 0\Mod{3}$ and $\textnormal{Re}(\zeta^{-a})=-1/2$ otherwise, using $\delta_0+\delta_1+\delta_2=1$, we can rewrite this as
$$ \frac{1}{27} \left( 3 \delta_{3-a} + o(1) \right).$$
Thus, in order to show that \eqref{eq31} is $\gg 1$, what remains to be shown is that $\delta_0,\delta_1,\delta_2\gg 1$. But since the sets $A_i$ are weakly stable and uniformly distributed with densities $1/3$ each, by Theorem \ref{theo1.5} we have
\begin{align*}
d_{-}\left( \bigcup_{\substack{c_1,c_2,c_3\in \{0,1,2\}\\c_1+c_2+c_3=c\Mod{3}}} (A_{c_1}-1)\cap (A_{c_2}-2) \cap (A_{c_3}-3)\right)>0,
\end{align*}
or in other words
\begin{align*}
d_{-}\left(\{n\in \mathbb{N}: \omega(n+1)+\omega(n+2)+\omega(n+3)\equiv c\Mod{3}\}\right)>0    
\end{align*}
for every $c\in \mathbb{Z}/3\mathbb{Z}$, which by partial summation implies $\delta_c>0$ for each $c$. The proof is now complete.
\end{proof}

\section{Sign patterns of the Liouville function}\label{sec: signpattern}

Before proving Theorem \ref{theo_sign}, we present a few lemmas. In what follows, $\omega(x)\leq x$ will be an arbitrary function tending to infinity.  By modifying the first part of the proof of Theorem \ref{weak-stable}, it suffices to show that
\begin{align*}
\limsup_{x\to \infty}\mathbb{E}_{x/\omega(x)\leq n\leq x}^{\log}1_{\lambda(n+1)=\varepsilon_1}\cdots 1_{\lambda(n+5)=\varepsilon_5}>0    
\end{align*}
for at least $24$ choices of $(\varepsilon_1,\ldots, \varepsilon_5)\in \{-1,+1\}^{5}$, and that the patterns listed in Theorem \ref{theo_sign} are among these $24$ patterns.

\begin{lemma} \label{le_isotopy} Let $k\geq 1$, and let $h_1,\ldots, h_k\in \mathbb{N}$. Let $1\leq \omega(X)\leq X$ be any function tending to infinity. Extend the Liouville function arbitrarily to negative integers. Then we have
\begin{align*}
\mathbb{E}_{x/\omega(x)\leq n\leq x}^{\log}\lambda(n+h_1)\cdots \lambda(n+h_k)=\mathbb{E}_{x/\omega(x)\leq n\leq x}^{\log}\lambda(n-h_1)\cdots \lambda(n-h_k)+o(1).    
\end{align*}
\end{lemma}

\begin{proof}
This is a direct corollary of the ''isotopy formula'' \cite[Theorem 1.2(iii)]{tt-elliott}.
\end{proof}

\begin{lemma}\label{le_12}
Let $k\geq 1$ be an integer, and let $1\leq \omega(X)\leq X$ be any function tending to infinity. Then we have
\begin{align*}
\limsup_{x\to \infty}\left|\mathbb{E}_{x/\omega(x)\leq n\leq x}^{\log}\lambda(n+1)\cdots \lambda(n+k)\right|\leq \frac{1}{2}.    
\end{align*}
\end{lemma}

\begin{proof}
This is a simple generalisation of \cite[Proposition 7.1]{tt-elliott}. By the triangle inequality, we have
\begin{align*}
&\left|\mathbb{E}_{x/\omega(x)\leq n\leq x}^{\log}\lambda(n+1)\cdots \lambda(n+k)+\lambda(n+2)\cdots \lambda(n+k+1)\right|\\
&\leq \mathbb{E}_{x/\omega(x)\leq n\leq x}^{\log}|\lambda(n+1)\cdots \lambda(n+k)+\lambda(n+2)\cdots \lambda(n+k+1)|\\
&= \mathbb{E}_{x/\omega(x)\leq n\leq x}^{\log}|\lambda(n+1)+\lambda(n+k+1)|.
\end{align*}
Here the first expression is equal to $2\left|\mathbb{E}_{x/\omega\leq n\leq x}^{\log}\lambda(n+1)\cdots \lambda(n+k)\right|+o(1)$ by the shift-invariance of logarithmic averages. But since $(\lambda(n+1),\lambda(n+k+1))$ takes each sign pattern in $\{-1,+1\}^2$ with density $1/4+o(1)$ with respect to the density $\mathbb{E}_{x/\omega\leq n\leq x}^{\log}$ by \cite[Theorem 1.2]{tao-chowla}, we get 
\begin{align*}
2\left|\mathbb{E}_{x/\omega(x)\leq n\leq x}^{\log}\lambda(n+1)\cdots \lambda(n+k)\right|\leq \frac{1}{2}|1+1|+\frac{1}{2}|1-1|+o(1)=1+o(1),
\end{align*}
as required.
\end{proof}

\begin{lemma} \label{le_correlationbound}
We have
\begin{align*}
\limsup_{x\to \infty}\left|\mathbb{E}_{x/\omega(x)\leq n\leq x}^{\log}\lambda(n+1)\lambda(n+2)\lambda(n+4)\lambda(n+5)\right|<1.    
\end{align*}
\end{lemma}

\begin{proof}
Suppose the contrary. Then there exists a sign $\varepsilon_0\in \{-1,+1\}$ and an infinite sequence $x_l\to \infty$ such that
\begin{align*}
\mathbb{E}_{x_l/\omega(x_l)\leq n\leq x}^{\log}\lambda(n+1)\lambda(n+2)\lambda(n+4)\lambda(n+5)=\varepsilon_0+o(1).    
\end{align*}
Consequently, we have
\begin{align}\label{eq14}
\mathbb{E}_{x_l/\omega(x_l)\leq n\leq x}^{\log}1_{\lambda(n+1)\lambda(n+2)\lambda(n+4)\lambda(n+5)=\varepsilon_0}=1+o(1). 
\end{align}
Shifting by one, we also have
\begin{align*}
\mathbb{E}_{x_l/\omega(x_l)\leq n\leq x}^{\log}1_{\lambda(n+2)\lambda(n+3)\lambda(n+5)\lambda(n+6)=\varepsilon_0}=1+o(1). 
\end{align*}
Putting the last two equations together, we obtain
\begin{align*}
\mathbb{E}_{x_l/\omega(x_l)\leq n\leq x}^{\log} 1_{\lambda(n+1)\lambda(n+3)\lambda(n+4)\lambda(n+6)=1}=1+o(1).
\end{align*}
Shifting by one again we have 
\begin{align}\label{eq15}
\mathbb{E}_{x_l/\omega(x_l)\leq n\leq x}^{\log} 1_{\lambda(n+2)\lambda(n+4)\lambda(n+5)\lambda(n+7)=1}=1+o(1).
\end{align}
 Finally, putting \eqref{eq14} and \eqref{eq15} together yields
\begin{align*}
\mathbb{E}_{x_l/\omega(x_l)\leq n\leq x}^{\log} 1_{\lambda(n+1)\lambda(n+7)=\varepsilon_0}=1+o(1),    
\end{align*}
and therefore
\begin{align*}
\mathbb{E}_{x/\omega(x)\leq n\leq x}^{\log }\lambda(n+1)\lambda(n+7)=\varepsilon_0 +o(1).    
\end{align*}
This however is in contradiction with the two-point logarithmic Chowla conjecture \cite[Theorem 1.2]{tt-chowla}.
\end{proof} 

\begin{proof}[Proof of Theorem \ref{theo_sign}] Let us define
\begin{align*}
C_{A}\coloneqq \plim \left(\mathbb{E}_{x_l/\omega(x_l)\leq n\leq x_l}^{\log}\prod_{j\in A}\lambda(n+j)\right)_{\ell\in \mathbb{N}},
\end{align*}
where $\plim$ is any generalised limit functional. Using the identity $1_{\lambda(n)=\varepsilon}=\frac{1+\varepsilon\lambda(n)}{2}$ for $\varepsilon\in \{-1,+1\}$ and expanding, we have
\begin{align}\label{eq16}
32\plim \left(\mathbb{E}_{x_l/\omega(x_l)\leq n\leq x_l}^{\log}1_{\lambda(n+1)=\varepsilon_1}\cdots 1_{\lambda(n+5)=\varepsilon_5}\right)_{\ell \in \mathbb{N}}=1+\sum_{\substack{A\subset [5]\\A\neq \emptyset}} C_A\prod_{j\in A}\varepsilon_j.  
\end{align}
It suffices to show that there are at least $24$ sign patterns $(\varepsilon_1,\ldots, \varepsilon_5)$ for which \eqref{eq16} is $>0$, regardless of which generalised limit $\plim$ we choose, including the $6$ explicit patterns listed in the theorem and their reversals.\footnote{It is this part of the argument that results in us obtaining a positive \emph{upper} density result rather than a positive \emph{lower} density result. Indeed, we show that for every generalized limit $\plim$ there are at least $24$ sign patterns $(\varepsilon_1,\ldots,\varepsilon_5)$ for which \eqref{eq16} is $>0$, but theoretically the choice of these $24$ sign patterns could depend on the choice of $\plim$, thus leading only to a $\limsup$ result. However, for each of the explicit patterns listed in Theorem \ref{theo_sign} we do obtain a lower density result by showing that \eqref{eq16} is always $>0$ for these sign patterns.}

By the odd order logarithmic Chowla conjecture \cite{tt-elliott}, all the odd order correlations are $0$, and by the two-point logarithmic Chowla conjecture \cite[Theorem 1.2]{tao-chowla}, all the two-point correlations are $0$ as well. Thus, if we denote the average on the left-hand side of \eqref{eq16} by $\mathbb{P}_{\varepsilon_1,\ldots, \varepsilon_5}$, then 
\begin{align*}
 32 \mathbb{P}_{\varepsilon_1,\ldots,\varepsilon_5}=1+\varepsilon_1\varepsilon_2\varepsilon_3\varepsilon_4\varepsilon_5(\varepsilon_1C_{[5]\setminus \{1\}}+\cdots +\varepsilon_5C_{[5]\setminus\{5\}}).  
\end{align*}
If we denote $C_{[5]\setminus \{1\}}\coloneqq a$, then by shift-invariance also $C_{[5]\setminus\{5\}}=a$. Furthermore, by Lemma \ref{le_isotopy}, if $C_{[5]\setminus \{2\}}=b$, then $C_{[5]\setminus \{4\}}=b$. Lastly, denote $C_{[5]\setminus \{3\}}=c$. We conclude that
\begin{align}\label{eq32}
 32 \mathbb{P}_{\varepsilon_1,\ldots,\varepsilon_5}=1+\varepsilon_1\varepsilon_2\varepsilon_3\varepsilon_4\varepsilon_5((\varepsilon_1+\varepsilon_5)a+(\varepsilon_2+\varepsilon_4)b+\varepsilon_3 c).    
\end{align}
Next, we split into several cases.\\

\emph{Case $a=b=0$.}  When this holds, by Lemma \ref{le_correlationbound} we have 
\begin{align*}
  32 \mathbb{P}_{\varepsilon_1,\ldots,\varepsilon_5}\geq 1-|c|>0   
\end{align*}
for each of the $32$ patterns.\\

\emph{Case $c\neq 0$, exactly one of $a,b\neq 0$.} Suppose that $a\neq 0, b=0$; the other case is symmetric. Then 
\begin{align*}
 32 \mathbb{P}_{\varepsilon_1,\ldots,\varepsilon_5}=1+\varepsilon_1\varepsilon_2\varepsilon_3\varepsilon_4\varepsilon_5
 ((\varepsilon_1+\varepsilon_5)a+\varepsilon_3c).
 \end{align*}
 Since $|a|\leq \frac{1}{2}$ and $|c|<1$ by Lemma \ref{le_12} and \ref{le_correlationbound}, respectively, the only way that the probability can be zero is if $\varepsilon_1=\varepsilon_5$, $\varepsilon_1\sgn(a)=\varepsilon_3\sgn(c)$ and $\varepsilon_1\varepsilon_2\varepsilon_4\varepsilon_5\sgn(c)=-1$. This happens for $32\cdot \frac{1}{2^3}=4$ sign patterns, so there are $32-4=28$ sign patterns having positive probability.\\

\emph{Case $c=0$, exactly one of $a,b\neq 0$.} Suppose that $a\neq 0,b=0$; the opposite case is symmetric. Then \begin{align*}
 32 \mathbb{P}_{\varepsilon_1,\ldots,\varepsilon_5}=1+(\varepsilon_1\varepsilon_2\varepsilon_3\varepsilon_4+\varepsilon_2\varepsilon_3\varepsilon_4\varepsilon_5)a,    
\end{align*}
and the only way this can be zero is if $\varepsilon_1\varepsilon_2\varepsilon_3\varepsilon_4=\varepsilon_2\varepsilon_3\varepsilon_4\varepsilon_5=-\sgn(a)$, which happens for exactly $32\cdot \frac{1}{2^2}=8$ patterns. Thus there are $32-8=24$ patterns with positive density. 

\emph{Case $c=0$, both $a,b\neq 0$.} Then we have 
\begin{align*}
 32 \mathbb{P}_{\varepsilon_1,\ldots,\varepsilon_5}=1+\varepsilon_1\varepsilon_2\varepsilon_3\varepsilon_4\varepsilon_5
 ((\varepsilon_1+\varepsilon_5)a+(\varepsilon_2+\varepsilon_4)b).
 \end{align*}
Now, consider $\varepsilon_i$ satisfying 
\begin{align*}
&\varepsilon_1\varepsilon_2\varepsilon_3\varepsilon_4\varepsilon_5=+1,\\
&\varepsilon_1=\varepsilon_5=-\sgn(a)\\
&\varepsilon_2=\varepsilon_4=-\sgn(b),
\end{align*}
which can always be found. The resulting probability is nonnegative, so
\begin{align*}
1-2|a|-2|b|\geq 0,    
\end{align*}
so $|a|+|b|\leq \frac{1}{2}$. Therefore, since $a,b\neq 0$, the only way that $\mathbb{P}_{\varepsilon_1,\ldots,\varepsilon_5}=0$ can happen is if $\varepsilon_1=\varepsilon_5$, $\varepsilon_2=\varepsilon_4$ and $\varepsilon_1\sgn(a)=\varepsilon_2\sgn(b)$. This happens for $32\cdot \frac{1}{2^3}=4$ patterns, so there must be at least $32-4=28$ patterns for which the probability is positive.

\emph{Case $a,b,c\neq 0$.} Now suppose that $\mathbb{P}_{\varepsilon_1,\ldots,\varepsilon_5}=0$, and consider the transformations of $(\varepsilon_1,\ldots, \varepsilon_5)$ given by
\begin{align*}
(\varepsilon_1, \varepsilon_2,\varepsilon_3,\varepsilon_4,\varepsilon_5)&\mapsto (-\varepsilon_1,\varepsilon_2,\varepsilon_3,\varepsilon_4,-\varepsilon_5)\\
(\varepsilon_1,\varepsilon_2, \varepsilon_3,\varepsilon_4, \varepsilon_5)&\mapsto (\varepsilon_1,-\varepsilon_2,\varepsilon_3,-\varepsilon_4,\varepsilon_5)\\
(\varepsilon_1, \varepsilon_2,\varepsilon_3,\varepsilon_4,\varepsilon_5)&\mapsto (-\varepsilon_1, -\varepsilon_2,\varepsilon_3,-\varepsilon_4,-\varepsilon_5). 
\end{align*}
Since $a\neq 0, b\neq 0$, each of the first two transformations changes the probability in \eqref{eq32}, in particular making it nonzero. The third transformation also changes the probability in \eqref{eq32}, unless $(\varepsilon_1+\varepsilon_5)a+(\varepsilon_2+\varepsilon_4)b=0$, in which case $32\mathbb{P}_{\varepsilon_1,\ldots, \varepsilon_5}=1+\varepsilon_3c>0$, contrary to our assumption. Thus, the patterns $(\varepsilon_1,\ldots, \varepsilon_5)$ can be grouped into groups of four where each group is closed under the above three transformations and has at most one pattern with zero probability. Hence, there are at least $32-\frac{32}{4}=24$ patterns having nonzero probability.\\

Since the above considerations exhaust all cases, we have now shown that there are at least $24$ sign patterns of length $5$ having positive upper density. We still need to show that the specific patterns mentioned in Theorem \ref{theo_sign} are among the patterns having positive upper density. The existence of the patterns having exactly one plus or exactly one minus follows directly from the proof strategy of \cite[Corollary 2.8]{mrt-sign} together with the fact that each length $4$ pattern occurs in the Liouville function with positive lower density. When it comes to the remaining patterns, consider $(+1,+1,\pm 1,-1,-1)$: the others are similar. This pattern has probability
\begin{align*}
32\mathbb{P}_{\varepsilon_1,\ldots,\varepsilon_5}=1\pm c>0    
\end{align*}
by Lemma \ref{le_correlationbound}. This completes the proof.\end{proof}

\bibliography{refs_valuepattern}
\bibliographystyle{plain}

\end{document}